\RequirePackage{fixltx2e}
\documentclass[oneside,reqno]{amsart}

\usepackage[latin9]{inputenc}
\usepackage{color}
\usepackage[english]{babel}
\usepackage{prettyref}
\usepackage{mathrsfs}
\usepackage{enumitem}
\usepackage{amsbsy}
\usepackage{amstext}
\usepackage{amsthm}
\usepackage{amssymb}
\usepackage{stmaryrd}
\usepackage{setspace}
\usepackage{xargs}[2008/03/08]
\usepackage{microtype}
\usepackage[unicode=true,pdfusetitle,
 bookmarks=true,bookmarksnumbered=false,bookmarksopen=false,
 breaklinks=false,pdfborder={0 0 1},backref=false,colorlinks=true]
 {hyperref}
\hypersetup{
 citecolor=blue,linkcolor=magenta,anchorcolor=green}

\makeatletter
\numberwithin{equation}{section}
\usepackage{changebar}
\providecommand{\lyxadded}[3]{}

\renewcommand{\lyxadded}[3]{
  {\protect\cbstart\color{lyxadded}{}#3\protect\cbend}
}

\theoremstyle{plain}
\newtheorem{lem}{\protect\lemmaname}
\theoremstyle{definition}
\newtheorem{defn}{\protect\definitionname}
\newcommand\thmsname{\protect\theoremname}
\newcommand\nm@thmtype{theorem}
\theoremstyle{plain}

\newenvironment{namedthm}[1][Undefined Theorem Name]{
  \ifx{#1}{Undefined Theorem Name}\renewcommand\nm@thmtype{theorem*}
  \else\renewcommand\thmsname{#1}\renewcommand\nm@thmtype{namedtheorem}
  \fi
  \begin{\nm@thmtype}}
  {\end{\nm@thmtype}}
\theoremstyle{plain}
\newtheorem{cor}{\protect\corollaryname}
\theoremstyle{plain}
\newtheorem{thm}{\protect\theoremname}
\theoremstyle{plain}
\newtheorem{conjecture}{\protect\conjecturename}
\theoremstyle{remark}
\newtheorem{rem}{\protect\remarkname}
\theoremstyle{plain}
\newtheorem{prop}{\protect\propositionname}

\@ifundefined{date}{}{\date{}}
\usepackage{upgreek}
\newrefformat{cor}{Corollary \ref{#1}}
\newrefformat{prop}{Proposition \ref{#1}}
\newrefformat{conj}{Conjecture \ref{#1}}
\newrefformat{nthm}{Theorem \nameref{#1}}
\newrefformat{def}{Definition \ref{#1}}

\makeatother

\providecommand{\conjecturename}{Conjecture}
\providecommand{\corollaryname}{Corollary}
\providecommand{\definitionname}{Definition}
\providecommand{\lemmaname}{Lemma}
\providecommand{\propositionname}{Proposition}
\providecommand{\remarkname}{Remark}
\providecommand{\theoremname}{Theorem}

\begin{document}
\selectlanguage{english}
\global\long\def\set#1#2{\left\{  #1\, |\, #2\right\}  }%

\global\long\def\cyc#1{\mathbb{Q}\!\left[\zeta_{#1}\right]}%

\global\long\def\mat#1#2#3#4{\left(\begin{array}{cc}
 #1  &  #2\\
 #3  &  #4 
\end{array}\right)}%

\global\long\def\Mod#1#2#3{#1\equiv#2\, \left(\mathrm{mod}\, \, #3\right)}%

\global\long\def\inv{^{\,\textrm{-}1}}%

\global\long\def\pd#1{#1^{+}}%

\global\long\def\fix#1{\mathtt{Fix}\!\left(#1\right)}%

\global\long\def\map#1#2#3{#1\!:\!#2\!\rightarrow\!#3}%

\global\long\def\Map#1#2#3#4#5{\begin{split}#1:#2  &  \rightarrow#3\\
 #4  &  \mapsto#5 
\end{split}
 }%

\global\long\def\partlat{\boldsymbol{\sqcap}}%

\global\long\def\fact#1#2{#1\slash#2}%

\global\long\def\Gal#1{\mathtt{Gal}\!\left(#1\right)}%

\global\long\def\fixf#1{\mathbb{Q}\!\left(#1\right)}%

\global\long\def\gl#1#2{\mathsf{GL}_{#2}\!\left(#1\right)}%

\global\long\def\SL{\mathrm{SL}_{2}\!\left(\mathbb{Z}\right)}%

\global\long\def\zn#1{\left(\mathbb{Z}/\!#1\mathbb{Z}\right)^{\times}}%

\global\long\def\sn#1{\mathbb{S}_{#1}}%

\global\long\def\aut#1{\mathrm{Aut\mathit{\left(#1\right)}}}%

\global\long\def\FA#1{\vert#1\vert}%

\global\long\def\FB#1{\mathtt{Z}^{2}\!(#1)}%

\global\long\def\FC#1{#1^{{\scriptscriptstyle \flat}}}%

\global\long\def\FD#1{#1^{{\scriptscriptstyle \times}}}%

\global\long\def\FE#1{\mathtt{x}_{#1}}%

\global\long\def\FF#1#2{\mathrm{Fix}_{#1}\left(#2\right)}%

\global\long\def\FI#1{#1_{{\scriptscriptstyle \pm}}}%

\global\long\def\cl#1{\mathscr{C}\negthinspace\ell\!\left(#1\right)}%

\global\long\def\bl#1{\mathcal{B}\ell\!\left(#1\right)}%

\global\long\def\bbl#1#2{\mathcal{B}\ell_{#2}\!\left(#1\right)}%

\global\long\def\FJ#1#2#3{\uptheta\!\Bigl[\!{#1\atop #2}\!\Bigr]\!\left(#3\right)}%

\global\long\def\ex#1{\mathtt{e}^{2\mathtt{i}\pi#1}}%

\newcommandx\exi[3][usedefault, addprefix=\global, 1=, 2=]{\mathtt{e}^{\mathtt{{\scriptscriptstyle \textrm{#1}}}\frac{#2\pi\mathtt{i}}{#3}}}%

\global\long\def\tw#1#2#3{\mathfrak{#1}_{#3}^{{\scriptscriptstyle \left(\!#2\!\right)}}}%

\global\long\def\ver#1{\mathtt{Ver}_{#1}}%

\global\long\def\stab#1#2{#1_{#2}}%

\global\long\def\cft{\mathscr{C}}%

\global\long\def\gal#1{\upsigma_{#1}}%

\global\long\def\galpi#1{#1}%

\global\long\def\ann#1{\ker\mathfrak{#1}}%

\global\long\def\cent#1#2{\mathtt{C}_{#1}\!\left(#2\right)}%

\global\long\def\ch#1{\boldsymbol{\uprho}_{#1}}%

\global\long\def\chb#1{\xi_{#1}}%

\global\long\def\chg{\mathfrak{g}}%

\global\long\def\tgr{G}%

\global\long\def\qd#1{\mathtt{d}_{#1}}%

\global\long\def\cw#1{\mathtt{h}_{#1}}%

\global\long\def\tcl{\textrm{trivial class}}%

\global\long\def\ccl{\mathtt{C}}%

\global\long\def\zcl{\mathtt{z}}%

\global\long\def\om#1{\omega\!\left(#1\right)}%

\global\long\def\rami#1{\mathtt{e}_{#1}}%

\global\long\def\zm{\mathtt{Z}}%

\global\long\def\cs#1{\left\llbracket #1\right\rrbracket }%

\global\long\def\coc#1{\vartheta_{\mathfrak{#1}}}%

\global\long\def\sp#1#2{\bigl\langle#1,#2\bigr\rangle}%

\global\long\def\irr#1{\mathtt{Irr}\!\left(#1\right)}%

\global\long\def\stb#1{\mathtt{I}{}_{#1}}%

\global\long\def\sqf#1{#1^{\circ}}%

\global\long\def\CHG{\textrm{twister}}%

\global\long\def\fc{\textrm{FC set}}%

\global\long\def\to#1{\boldsymbol{\upsigma}\!\left(#1\right)}%

\global\long\def\v{\mathtt{{\scriptstyle 0}}}%

\global\long\def\wm#1{\mathcal{P}\!\left(#1\right)}%

\global\long\def\fm{\mathtt{N}}%

\global\long\def\du#1{#1^{{\scriptscriptstyle \perp}}}%

\global\long\def\dhg{\du{\chg}}%

\global\long\def\spr{\textrm{spread}}%

\global\long\def\gcl{\textrm{class}}%

\global\long\def\dcl{\textrm{block}}%

\global\long\def\vorb{\mathfrak{o}}%

\global\long\def\lat{\mathcal{\mathscr{L}}}%

\global\long\def\vera{\mathcal{V}}%

\global\long\def\svera#1{\vera_{#1}}%

\global\long\def\bd{\mathcal{D}_{\mathfrak{b}}}%

\global\long\def\rd#1{\boldsymbol{\updelta}_{#1}}%

\global\long\def\m#1{\mathfrak{\upmu}_{\mathfrak{#1}}}%

\global\long\def\zent#1{\mathtt{Z}\!\left(#1\right)}%

\global\long\def\voa{\mathcal{\mathbb{V}}}%

\global\long\def\bfm{\Delta}%

\global\long\def\cech#1#2{\boldsymbol{\varpi}_{#1}\!\left(#2\right)}%

\global\long\def\clchar#1#2{\updelta_{#1}\!\left(#2\right)}%

\global\long\def\vb{\mathfrak{b}_{{\scriptscriptstyle 0}}}%

\global\long\def\bch#1#2{\boldsymbol{\upchi}_{\mathfrak{#1}}\!\left(#2\right)}%

\global\long\def\prim{\mathtt{X}}%

\global\long\def\sig{\FA{\cft}}%

\global\long\def\beq#1{\boldsymbol{\Lambda}_{#1}}%

\global\long\def\twa#1{\boldsymbol{\nabla}\!{}_{#1}}%

\global\long\def\dg#1{\hat{#1}}%

\global\long\def\loc{\lat_{\mathtt{loc}}}%
\global\long\def\intlat{\lat_{\mathtt{int}}}%

\global\long\def\spp#1#2{\mathcal{P}_{#1}\!\left(#2\right)}%

\title{FC sets and twisters: the basics of orbifold deconstruction}
\author{Peter Bantay}
\curraddr{Institute for Theoretical Physics, E\" otv\" os Lor\' and University,  H-1117 Budapest, P\' azm\' any P\' eter s. 1/A}
\keywords{conformal symmetry, orbifold models, modular tensor categories, character
rings}
\begin{abstract}
We present a detailed account of the properties of $\CHG$s and their
generalizations, $\fc$s, which are essential ingredients of the orbifold
deconstruction procedure aimed at recognizing whether a given conformal
model may be obtained as an orbifold of another one, and if so, to
identify the twist group and the original model. The close analogy
with the character theory of finite groups is discussed, and its origin
explained.
\end{abstract}

\maketitle

\section{Introduction}

Orbifold deconstruction, i.e. the procedure aimed at recognizing whether
a given 2D conformal model is an orbifold \cite{Dixon_orbifoldCFT,FLM1}
of another one, and if so, to identify (up to isomorphism) the relevant
twist group and the original model, is an effective tool to better
understand both the general properties of conformal models and the
precise structure of their orbifolds. The basic ideas have been described
in \cite{Bantay2018a,Bantay2018b}, focusing on conceptual issues
without going into the mathematical details. The purpose of the present
paper is to fill this gap by giving a formal treatment of the concepts
underlying the deconstruction procedure.

The starting point of orbifold deconstruction is the observation \cite{Bantay2018a,Bantay2018b}
that every orbifold has a distinguished set of primaries, the so-called
vacuum block, consisting of the descendants of the vacuum, and that
this vacuum block has quite special properties: it is closed under
the fusion product, and all its elements have integral conformal weight
and quantum dimension. Such sets of primaries were termed 'twisters'
because of their relation to twist groups and twisted boundary conditions.
Twisters provide the input for the deconstruction procedure: to each
different twister corresponds a different deconstruction, with possibly
different twist groups and/or deconstructed models.

It turns out that most properties of twisters can be understood in
the more general context of $\fc$s, which are those sets of primaries
that are closed under the fusion product. As we shall see, these show
deep analogies with character rings of finite groups, especially the
so-called integral $\fc$s, all of whose elements have integral squared
quantum dimension. In case of twisters, this analogy with character
theory is of course far from being accidental, for it stems from their
relation with the twist group of the corresponding orbifold, and it
allows the generalization of several important group theoretic notions
(like nilpotency, solubility, etc.) to general $\fc$s. In this respect,
a most interesting question is: to what extent do classical results
about groups generalize to selected classes of $\fc$s? We shall encounter
several such conjectural generalizations on the way, e.g. of Lagrange's
and Ito's celebrated theorems.

It should be pointed out that a special class of $\fc$s, the Abelian
ones (cf. \prettyref{def:Abel}) have been well-known for quite some
time \cite{SY1,SY2,Intriligator}, their elements running under the
name of simple currents, while the corresponding deconstructions are
known as simple current extension \cite{Fuchs1996a,Bantay1998}. From
this point of view it is fair to say that $\fc$s could be viewed
as non-Abelian generalizations of simple current groups, and their
theory bears the same relationship with that of simple currents as
the representation theory of groups \cite{fulton-harris,Alperin-Bell,Lux-Pahlings}
with that of Abelian ones.

In \prettyref{sec:Preliminaries} we'll review those standard results
about the fusion ring of rational conformal models that form the basis
of most of the subsequent arguments. \prettyref{sec:FC} develops
the basic theory of $\fc$s, introducing such fundamental notions
as classes, blocks and their overlaps, and proving the modularity
of the lattice of $\fc$s. \prettyref{sec:The-center} introduces
the center of an $\fc$, and describes its basic properties, while
the next section deals with central quotients and extensions. \prettyref{sec:The-Galois-action}
is concerned with the arithmetic properties of $\fc$s, while \prettyref{sec:Local-sets-and}
describes the structure of local $\fc$s and twisters, with a view
towards their role in orbifold deconstruction. In our opinion, the
highlights include, besides the orthogonality relations Eqs.\prettyref{eq:ortho1}
and \prettyref{eq:ortho2}, the product rule \prettyref{thm: Product rule},
\prettyref{thm:modlat} on the modularity of the lattice of $\fc$s,
\prettyref{thm:zentextstruct} on the structure of central quotients,
and \prettyref{thm:main} on the integrality of quantum dimensions
in local $\fc$s.

\section{Preliminaries\label{sec:Preliminaries}}

Let's consider a rational unitary conformal model \cite{DiFrancesco-Mathieu-Senechal,Ginsparg1988,Moore.1990}.
We'll denote by $\qd p$ and $\mathtt{h}_{p}$ the quantum dimension
and conformal weight of a primary $p$, and by $\fm\!\left(p\right)$
the associated fusion matrix, whose matrix elements are given by the
fusion rules
\begin{equation}
\left[\fm\!\left(p\right)\right]_{q}^{r}=N_{pq}^{r}\label{eq:fmdef}
\end{equation}
We'll denote by $\v$ the vacuum primary, for which $\qd{\v}\!=\!1$,
$\cw{\v}\!=\!0$ and $\fm\!\left(\v\right)$ is the identity matrix.
Note that, since
\begin{equation}
\fm\!\left(p\right)\fm\!\left(q\right)=\sum_{r}N_{pq}^{r}\fm\!\left(r\right)\label{ver1}
\end{equation}
the fusion matrices generate a commutative matrix algebra over $\mathbb{C}$,
the Verlinde algebra $\vera$, whose irreducible representations are
all of dimension $1$. According to Verlinde's famous formula \cite{Verlinde1988}
\begin{equation}
N_{pq}^{r}=\sum_{w}\frac{S_{pw}S_{qw}\overline{S_{rw}}}{S_{\v w}}\label{eq:ver2}
\end{equation}
relating the fusion rules to the modular $S$-matrix, to each primary
$p$ corresponds an irrep $\ch p$ of $\vera$ that assigns to the
fusion matrix $\fm\!\left(q\right)$ the complex number
\begin{equation}
\ch p\!\left(q\right)\!=\!\frac{S_{qp}}{S_{\v p}}\label{eq:chdef}
\end{equation}
In view of Eq.\eqref{ver1} this gives
\begin{equation}
\sum_{r}N_{pq}^{r}\ch w\!\left(r\right)\!=\!\ch w\!\left(p\right)\ch w\!\left(q\right)\label{eq:verrep}
\end{equation}
which is equivalent to
\begin{equation}
\sum_{r}N_{pq}^{r}S_{rw}=\frac{S_{pw}S_{qw}}{S_{\v w}}\label{eq:ver4}
\end{equation}
The quantum dimensions of the primaries, i.e. the common Perron-Frobenius
eigenvector \cite{Gantmacher1959} of the fusion matrices are given
by $\qd p\!=\!\ch{\v}\!\left(p\right)$, and one has the inequality
\begin{equation}
\FA{\ch p\!\left(q\right)}\!\leq\!\qd q\label{eq:ineq}
\end{equation}
Since the fusion matrices have integer matrix elements, it follows
that their eigenvalues, i.e. the $\ch p\!\left(q\right)$ are all
algebraic integers; in particular, all quantum dimensions $\qd p$
are algebraic integers, that may be shown to divide the algebraic
integer $\sqrt{\sum_{p}\qd p^{2}}\!=\!S_{\v\v}\inv$.

Note that the matrix $S$ used above is actually an auxiliary quantity,
since it can be determined fully from the fusion rules and conformal
weights through the formula
\begin{equation}
\frac{S_{pq}}{S_{\v\v}}=\sum_{r}N_{pq}^{r}\qd r\frac{\om p\om q}{\om r}\label{eq:ver3}
\end{equation}
where
\begin{equation}
\om p=\exp\!\left(2\pi\mathtt{i}\cw p\right)\label{eq:omdef}
\end{equation}
is the exponentiated conformal weight of the primary $p$.
\begin{lem}
\begin{singlespace}
\label{lem:omch}For primaries $p$ and $q$ such that $\FA{\ch p\!\left(q\right)\!}\!=\!\qd q$,
$N_{pq}^{r}\!>\!0$ iff
\begin{equation}
\frac{\om p\om q}{\om r}\!=\!\frac{\ch p\!\left(q\right)}{\qd q}\label{eq:omch}
\end{equation}
\end{singlespace}
\end{lem}
\begin{proof}
If Eq.\eqref{eq:omch} holds, then obviously $\FA{\ch p\!\left(q\right)\!}\!=\!\qd q$,
since the left-hand side has modulus $1$. Conversely, according to
Eq.\eqref{eq:ver3}
\[
\sum_{r}N_{pq}^{r}\qd r\frac{\om p\om q}{\om r}\!=\!\frac{S_{pq}}{S_{\v\v}}\!=\!\frac{S_{pq}}{S_{\v p}}\frac{S_{\v p}}{S_{\v\v}}\!=\!\ch p\!\left(q\right)\!\qd p\!=\!\frac{\ch p\!\left(q\right)}{\qd q}\sum_{r}N_{pq}^{r}\qd r
\]
hence for $\ch p\!\left(q\right)\!\neq\!0$ one obtains
\[
0\!=\!\sum_{r}N_{pq}^{r}\qd r\!\left(1\!-\!\frac{\om p\om q\qd q}{\om r\ch p\!\left(q\right)}\right)
\]
or, after taking real parts
\[
\sum_{r}N_{pq}^{r}\qd r\!\left(1\!-\!\mathtt{Re}\left(\frac{\om p\om q\qd q}{\om r\ch p\!\left(q\right)}\right)\right)=0
\]
Since the real part of a complex number cannot exceed its modulus
\[
\mathtt{Re}\left(\frac{\om p\om q\qd q}{\om r\ch p\!\left(q\right)}\right)\leq\left|\frac{\om p\om q\qd q}{\om r\ch p\!\left(q\right)}\right|=1
\]
for $\FA{\ch p\!\left(q\right)\!}\!=\!\qd q\!>\!0$, consequently
all terms of the sum are non-negative, hence they should all vanish,
proving the claim.
\end{proof}

\section{FC sets\label{sec:FC}}
\begin{defn}
A set $\mathfrak{g}$ of primaries is fusion closed (an $\mathcal{\fc}$
for short) if it contains the vacuum primary $\v$, and for all $\alpha,\beta\!\in\!\chg$
\begin{equation}
\sum_{\gamma\in\chg}N_{\alpha\beta}^{\gamma}\qd{\gamma}=\qd{\alpha}\qd{\beta}\label{eq:fccrit}
\end{equation}
\end{defn}
Taking into account that quantum dimensions are positive real numbers,
this is tantamount to the requirement that $N_{\alpha\beta}^{\gamma}\!>\!0$
and $\alpha,\beta\!\in\!\chg$ implies $\gamma\!\in\!\chg$. Notice
that
\begin{equation}
\fm\!\left(\alpha\right)\fm\!\left(\beta\right)=\sum_{\gamma\in\mathfrak{g}}\!N_{\alpha\beta}^{\gamma}\fm\!\left(\gamma\right)\label{eq:vgdef}
\end{equation}
for $\alpha,\beta\!\in\!\mathfrak{g}$ by Eq.\eqref{ver1}, hence
the fusion matrices $\fm\!\left(\alpha\right)$ generate a subalgebra
$\svera{\mathfrak{g}}$ of the Verlinde algebra. Since $\vera$ is
commutative, the irreps of the subalgebra $\svera{\mathfrak{g}}$
are among the different restrictions of the irreps $\ch p$ of $\vera$.
\begin{defn}
Given an $\fc$ $\mathfrak{g}$, a $\mathfrak{g}$-$\gcl$ $\ccl$
is the set of all those primaries $p$ whose associated irreps $\ch p$
coincide when restricted to the subalgebra $\svera{\mathfrak{g}}$;
we shall denote by $\ch{\ccl}$ this common restriction, and by $\alpha\!\left(\ccl\right)\!=\!\ch{\ccl}\!\left(\alpha\right)$
the value it assigns to an element $\alpha\!\in\!\chg$.

Clearly, the collection $\cl{\mathfrak{g}}$ of $\mathfrak{g}$-$\gcl$es
is a partition of the set of all primaries, and one has $S_{\alpha p}\!=\!\alpha\!\left(\ccl\right)\!S_{\v p}$
for $\alpha\!\in\!\mathfrak{g}$ if the primary $p$ belongs to the
class $\ccl\!\in\!\cl{\mathfrak{g}}$.
\end{defn}
\begin{lem}
\begin{singlespace}
The number of $\mathfrak{g}$-$\gcl$es equals the cardinality of
$\chg$: 
\begin{equation}
\FA{\cl{\mathfrak{g}}\!}=\FA{\mathfrak{g}}\label{eq:classno}
\end{equation}
\end{singlespace}
\end{lem}
\begin{proof}
As explained above, the irreps of the subalgebra $\svera{\mathfrak{g}}$
are among the different restrictions of the irreps of $\vera$, i.e.
the irreps $\ch{\ccl}$ corresponding to the $\gcl$es. It follows
that the number of $\gcl$es equals the dimension (over $\mathbb{C}$)
of $\svera{\mathfrak{g}}$, and the later equals the cardinality of
$\chg$, because the fusion matrices are linearly independent.
\end{proof}
\begin{defn}
The extent $\cs{\ccl}$ of the class $\ccl\!\in\!\cl{\mathfrak{g}}$
is the algebraic number
\begin{equation}
\cs{\ccl}=\frac{1}{\sum\limits _{p\in\ccl}S_{\v p}^{2}}\label{eq:extdef}
\end{equation}
\end{defn}
\begin{lem}
\begin{singlespace}
\begin{equation}
\sum_{\ccl\in\cl{\mathfrak{g}}}\frac{1}{\cs{\ccl}}=1\label{eq:extsumrule}
\end{equation}
\end{singlespace}
\end{lem}
\begin{proof}
This follows at once from the unitarity of the matrix $S$.
\end{proof}
\begin{namedthm}[Orthogonality relations]
\begin{singlespace}
For $\alpha,\beta\!\in\!\mathfrak{g}$
\begin{equation}
\sum_{\ccl\in\cl{\mathfrak{g}}}\!\frac{\alpha\!\left(\ccl\right)\overline{\beta\!\left(\ccl\right)}}{\cs{\ccl}}=\begin{cases}
1 & \textrm{\emph{if} }\alpha\!=\!\beta;\\
0 & \mathcal{\textrm{\emph{otherwise}}}.
\end{cases}\label{eq:ortho1}
\end{equation}
\end{singlespace}
\end{namedthm}
\begin{proof}
\[
\sum_{\ccl\in\cl{\mathfrak{g}}}\frac{\alpha\!\left(\ccl\right)\overline{\beta\!\left(\ccl\right)}}{\cs{\ccl}}=\sum_{\ccl\in\cl{\mathfrak{g}}}\!\sum_{p\in\ccl}S_{\v p}^{2}\frac{S_{\alpha p}}{S_{\v p}}\frac{\overline{S_{\beta p}}}{S_{\v p}}=\sum_{p}S_{\alpha p}\overline{S_{p\beta}}=\delta_{\alpha,\beta}
\]
by the unitarity of the matrix $S$.
\end{proof}
\begin{namedthm}[Second orthogonality]
\begin{singlespace}
For any two $\gcl$es $\ccl_{1},\ccl_{2}\!\in\!\cl{\mathfrak{g}}$
\begin{equation}
\sum_{\alpha\in\chg}\alpha\!\left(\ccl_{1}\right)\overline{\alpha\!\left(\ccl_{2}\right)}=\begin{cases}
\cs{\ccl_{1}} & \textrm{if }\ccl_{1}\!=\!\ccl_{2};\\
0 & \textrm{otherwise}.
\end{cases}\label{eq:ortho2}
\end{equation}
\end{singlespace}
\end{namedthm}
\begin{proof}
Consider the square matrix
\[
X_{\alpha\ccl}=\frac{\alpha\!\left(\ccl\right)}{\sqrt{\cs{\ccl}}}
\]
with rows indexed by the elements $\alpha\!\in\!\chg$ and columns
by the $\gcl$es $\ccl\!\in\!\cl{\mathfrak{g}}$. But
\[
\sum_{\ccl\in\cl{\mathfrak{g}}}X_{\alpha\ccl}\overline{X_{\beta\ccl}}=\delta_{\alpha,\beta}
\]
by Eq.\eqref{eq:ortho1}, meaning that the matrix $X$ is unitary,
which implies at once
\[
\sum_{\alpha\in\chg}\frac{\overline{\alpha\!\left(\ccl_{1}\right)}\alpha\!\left(\ccl_{2}\right)}{\sqrt{\cs{\ccl_{1}}\cs{\ccl_{2}}}}=\sum_{\alpha\in\chg}\overline{X_{\alpha\ccl_{1}}}X_{\alpha\ccl_{2}}=\begin{cases}
1 & \textrm{if }\ccl_{1}\!=\!\ccl_{2};\\
0 & \textrm{otherwise}.
\end{cases}
\]
\end{proof}
\begin{cor}
\begin{singlespace}
The cardinality of the class $\ccl\!\in\!\cl{\mathfrak{g}}$ is given
by
\begin{equation}
\FA{\ccl}=\frac{1}{\cs{\ccl}}\sum_{\alpha\in\chg}\overline{\alpha\!\left(\ccl\right)}\mathrm{Tr}\,\fm\!\left(\alpha\right)\label{eq:classsize}
\end{equation}
\end{singlespace}
\end{cor}
\begin{proof}
\begin{singlespace}
Since the trace of a matrix equals the sum of its eigenvalues, 
\[
\mathrm{Tr}\,\fm\!\left(\alpha\right)=\sum_{p}\ch p\!\left(\alpha\right)=\sum_{\ccl\in\cl{\mathfrak{g}}}\sum_{p\in\ccl}\ch p\!\left(\alpha\right)=\sum_{\ccl\in\cl{\mathfrak{g}}}\!\!\FA{\ccl}\alpha\!\left(\ccl\right)
\]
\end{singlespace}

\noindent for $\alpha\!\in\!\mathfrak{g}$, and the result follows
at once from \prettyref{eq:ortho2}.
\end{proof}
\begin{lem}
\begin{singlespace}
The characteristic function of the $\gcl$ $\ccl\!\in\!\cl{\mathfrak{g}}$
reads
\begin{equation}
\clchar{\ccl}p=\frac{1}{\cs{\ccl}}\sum_{\alpha\in\chg}\overline{\alpha\!\left(\ccl\right)}\frac{S_{\alpha p}}{S_{\v p}}=\begin{cases}
1 & \textrm{\emph{if} }p\!\in\!\ccl;\\
0 & \textrm{\emph{otherwise}.}
\end{cases}\label{eq:classcharfun}
\end{equation}
\end{singlespace}
\end{lem}
\begin{proof}
The primary $p$ belongs to the class $\mathtt{D}\!\in\!\cl{\mathfrak{g}}$
iff $S_{\alpha p}\!=\!\alpha\!\left(\mathtt{D}\right)\!S_{\v p}$
for $\alpha\!\in\!\mathfrak{g}$, hence
\[
\frac{1}{\cs{\ccl}}\sum_{\alpha\in\chg}\overline{\alpha\!\left(\ccl\right)}\frac{S_{\alpha p}}{S_{\v p}}\!=\!\frac{1}{\cs{\ccl}}\sum_{\alpha\in\chg}\overline{\alpha\!\left(\ccl\right)}\alpha\!\left(\mathtt{D}\right)\!=\!\begin{cases}
1 & \textrm{if }\mathtt{D}\!=\!\ccl;\\
0 & \textrm{otherwise.}
\end{cases}
\]

\noindent by Eq.\prettyref{eq:ortho2}, proving the claim.
\end{proof}
\begin{lem}
\begin{singlespace}
\label{lem:wmelms}For any $\gcl$ $\ccl\!\in\!\cl{\mathfrak{g}}$
one has
\begin{equation}
\sum_{w\in\ccl}S_{pw}\overline{S_{wq}}=\frac{1}{\cs{\ccl}}\sum_{\alpha\in\chg}\overline{\alpha\!\left(\ccl\right)}N_{\alpha p}^{q}\label{eq:wmatelms}
\end{equation}
\end{singlespace}
\end{lem}
\begin{proof}
\begin{singlespace}
It follows from Eqs. \eqref{eq:ver2} and \eqref{eq:classcharfun}
that
\begin{gather*}
\sum_{w\in\ccl}S_{pw}\overline{S_{wq}}=\sum_{w}\clchar{\ccl}wS_{pw}\overline{S_{wq}}=\sum_{w}S_{pw}\overline{S_{qw}}\Bigl\{\!\frac{1}{\cs{\ccl}}\!\sum_{\alpha\in\chg}\overline{\alpha\!\left(\ccl\right)}\frac{S_{\alpha w}}{S_{\v w}}\Bigr\}\\
=\frac{1}{\cs{\ccl}}\!\sum_{\alpha\in\chg}\overline{\alpha\!\left(\ccl\right)}\Bigl\{\sum_{w}\frac{S_{pw}\overline{S_{qw}}S_{\alpha w}}{S_{\v w}}\Bigr\}=\frac{1}{\cs{\ccl}}\!\sum_{\alpha\in\chg}\overline{\alpha\!\left(\ccl\right)}N_{\alpha p}^{q}
\end{gather*}
\end{singlespace}
\end{proof}
The $\gcl$ containing the vacuum primary $\v$ is of special importance:
we shall denote it by $\du{\chg}$, and call it the $\tcl$. According
to the previous definitions, $\alpha\!\left(\du{\chg}\right)\!=\!\ch{\v}\!\left(\alpha\right)\!=\!\qd{\alpha}$
for $\alpha\!\in\!\chg$.
\begin{lem}
\begin{singlespace}
\begin{equation}
\cs{\du{\chg}}=\sum_{\alpha\in\chg}\qd{\alpha}^{2}\label{eq:spread}
\end{equation}
\end{singlespace}
\end{lem}
\begin{proof}
\noindent Since $\alpha\!\left(\du{\mathfrak{g}}\right)\!=\!\qd{\alpha}$
for $\alpha\!\in\!\mathfrak{g}$, one has
\[
\sum_{\alpha\in\mathfrak{g}}\qd{\alpha}^{2}\!=\!\sum_{\alpha\in\mathfrak{g}}\overline{\alpha\!\left(\du{\mathfrak{g}}\right)}\alpha\!\left(\du{\mathfrak{g}}\right)\!=\!\cs{\du{\chg}}
\]
according to Eq.\prettyref{eq:ortho2}.
\end{proof}
The trivial class maximizes the product of size and extent.
\begin{lem}
\label{lem:sizebound2} $\FA{\ccl}\!\cs{\ccl}\!\leq\!\FA{\du{\mathfrak{g}}}\!\cs{\du{\mathfrak{g}}}$
for every class $\ccl\!\in\!\cl{\mathfrak{g}}$.
\end{lem}
\begin{proof}
\begin{singlespace}
By Eq.\prettyref{eq:classsize}
\[
\FA{\ccl}\!\cs{\ccl}=\sum_{\alpha\in\chg}\overline{\alpha\!\left(\ccl\right)}\mathrm{Tr}\,\fm\!\left(\alpha\right)
\]
Since the matrix $\fm\!\left(\alpha\right)$ is non-negative and $\FA{\alpha\!\left(\ccl\right)\!}\!\leq\!\qd{\alpha}\!=\!\alpha\!\left(\du{\mathfrak{g}}\right)$
\[
\FA{\ccl}\!\cs{\ccl}\!\leq\!\sum_{\alpha\in\chg}\FA{\overline{\alpha\!\left(\ccl\right)}}\,\mathrm{Tr}\,\fm\!\left(\alpha\right)\!\leq\!\sum_{\alpha\in\chg}\alpha\!\left(\du{\mathfrak{g}}\right)\mathrm{Tr}\,\fm\!\left(\alpha\right)\!=\!\FA{\du{\mathfrak{g}}}\!\cs{\du{\mathfrak{g}}}
\]
\end{singlespace}

\noindent by the triangle inequality, taking into account that $\FA{\ccl}\!\cs{\ccl}\!>\!0$.
\end{proof}
\begin{thm}[Product rule]
\label{thm: Product rule} If $N_{pq}^{r}\!>\!0$ for some $p\!\in\!\du{\chg}$,
then the primaries $q$ and $r$ belong to the same $\mathfrak{g}$-$\gcl$.
\end{thm}
\begin{proof}
\begin{singlespace}
Denoting by $\ccl$ the $\gcl$ of $q$, one has the obvious equality
\[
\sum_{r\notin\ccl}\!N_{pq}^{r}\qd r=\sum_{r}\!N_{pq}^{r}\qd r-\sum_{r\in\ccl}\!N_{pq}^{r}\qd r=\qd p\qd q-\sum_{r\in\ccl}\!N_{pq}^{r}\qd r
\]
On the other hand, by Eqs.\prettyref{eq:classcharfun} and \prettyref{eq:verrep}
\begin{gather*}
\sum_{r\in\ccl}\!N_{pq}^{r}\qd r=\sum_{r}\!\clchar{\ccl}r\!N_{pq}^{r}\qd r=\sum\!N_{pq}^{r}\frac{S_{\v r}}{S_{\v\v}}\!\left\{ \!\frac{1}{\cs{\ccl}}\!\sum_{\alpha\in\chg}\!\overline{\alpha\!\left(\ccl\right)}\frac{S_{\alpha r}}{S_{\v r}}\!\right\} \!\\
=\frac{1}{\cs{\ccl}}\!\sum_{\alpha\in\chg}\!\overline{\alpha\!\left(\ccl\right)}\frac{S_{\v\alpha}}{S_{\v\v}}\!\left\{ \!\sum_{r}\!N_{pq}^{r}\frac{S_{\alpha r}}{S_{\v\alpha}}\!\right\} =\frac{1}{\cs{\ccl}}\!\sum_{\alpha\in\chg}\!\overline{\alpha\!\left(\ccl\right)}\frac{S_{\v\alpha}}{S_{\v\v}}\frac{S_{\alpha p}}{S_{\v\alpha}}\frac{S_{\alpha q}}{S_{\v\alpha}}
\end{gather*}
Since $S_{\alpha p}\!=\!\qd{\alpha}S_{\v p}\!=\!\qd pS_{\v\alpha}$
for $p\!\in\!\du{\chg}$ and $S_{\alpha q}\!=\!\alpha\!\left(\ccl\right)\!S_{\v q}$,
this gives
\[
\sum_{r\in\ccl}\!N_{pq}^{r}\qd r\!=\!\frac{1}{\cs{\ccl}}\!\sum_{\alpha\in\chg}\!\overline{\alpha\!\left(\ccl\right)}\frac{S_{\v\alpha}}{S_{\v\v}}\frac{S_{\alpha p}}{S_{\v\alpha}}\frac{S_{\alpha q}}{S_{\v\alpha}}\!=\!\qd p\frac{S_{\v q}}{S_{\v\v}}\frac{1}{\cs{\ccl}}\!\sum_{\alpha\in\chg}\!\overline{\alpha\!\left(\ccl\right)}\alpha\!\left(\ccl\right)\!=\!\qd p\qd q
\]
from which one concludes
\[
\sum_{r\notin\ccl}\!N_{pq}^{r}\qd r=0
\]
\end{singlespace}

\noindent Since all terms on the left-hand side are non-negative,
it follows that they all have to vanish, i.e. $N_{pq}^{r}\!=\!0$
for $r\notin\ccl$.
\end{proof}
\begin{cor}
\label{cor:duality}$\du{\chg}$ is an $\fc$.
\end{cor}
\begin{proof}
If $p,q\!\in\!\du{\mathfrak{g}}$ and $N_{pq}^{r}\!>\!0$, then $r\!\in\!\du{\mathfrak{g}}$
by \prettyref{thm: Product rule}.
\end{proof}
\prettyref{cor:duality} implies that all notions and results about
an $\fc$ $\mathfrak{g}$ go over verbatim to its dual $\fc$ $\du{\chg}$.
In particular, the set of primaries is partitioned into $\du{\mathfrak{g}}$-$\gcl$es,
which we shall call $\mathfrak{g}$-blocks (or simply blocks) to avoid
confusion with $\mathfrak{g}$-$\gcl$es.
\begin{defn}
For an $\fc$ $\mathfrak{g}\!\in\!\lat$, a $\mathfrak{g}$-block
is a class of the dual $\fc$ $\du{\mathfrak{g}}$. We'll denote the
collection $\cl{\du{\mathfrak{g}}}$ of $\mathfrak{g}$-blocks by
$\bl{\mathfrak{g}}$.
\end{defn}
\begin{lem}
\label{lem:blcrit2}The primaries $p$ and $q$ belong to the same
$\mathfrak{g}$-$\dcl$ iff there exists $\alpha\!\in\!\chg$ such
that $N_{\alpha p}^{q}\!>\!0$.
\end{lem}
\begin{proof}
Since $\du{\chg}$ is an $\fc$ according to \prettyref{cor:duality},
the orthogonality relations apply to it. In particular, Eq.\eqref{eq:ortho2}
takes the form
\[
\sum_{w\in\dhg}\!\frac{S_{wp}}{S_{\v p}}\frac{\overline{S_{wq}}}{S_{\v q}}\!=\!\begin{cases}
\!\cs{\mathfrak{b}} & \!\!\textrm{if \ensuremath{p} and \ensuremath{q} belong to the same \ensuremath{\dcl} \ensuremath{\mathfrak{b}\!\in\!\bl{\mathfrak{g}}};}\\
\!0 & \!\!\textrm{otherwise. }
\end{cases}
\]
By \prettyref{lem:wmelms}, this means that $p$ and $q$ belong to
the same $\ensuremath{\dcl}$ precisely when
\[
\sum_{w\in\dhg}\!S_{wp}\overline{S_{wq}}\!=\!\frac{1}{\cs{\du{\mathfrak{g}}}}\sum_{\alpha\in\chg}\qd{\alpha}N_{\alpha p}^{q}>0
\]
Since the quantum dimensions $\qd{\alpha}$ are all positive, this
is equivalent to $N_{\alpha p}^{q}\!>\!0$ for some $\alpha\!\in\!\chg$.
\end{proof}
\begin{cor}
The $\mathfrak{g}$-block containing the vacuum is $\mathfrak{g}$
itself: $\du{\left(\du{\chg}\!\right)\!}\!=\!\chg$.
\end{cor}
\begin{proof}
Indeed, if $q$ belongs to the same block as the vacuum primary $\v$,
then there exist $\alpha\!\in\!\chg$ such that $\delta_{q,\alpha}\!=\!N_{\alpha\v}^{q}\!>\!0$
by \prettyref{lem:blcrit2}.
\end{proof}
\begin{lem}
\label{lem:spreadrecip}
\begin{equation}
\cs{\chg}\!\cs{\du{\chg}}=\sum_{p}\qd p^{2}\label{eq:recip}
\end{equation}
\end{lem}
\begin{proof}
\begin{singlespace}
\noindent By Eqs.\prettyref{eq:spread} and \prettyref{eq:classcharfun}
\begin{gather*}
\cs{\mathfrak{g}}\!=\!\sum_{w\in\du{\mathfrak{g}}}\!\qd w^{2}\!=\!\sum_{w}\clchar{\du{\mathfrak{g}}}w\qd w^{2}\!=\!\sum_{w}\frac{1}{\cs{\du{\mathfrak{g}}}}\sum_{\alpha\in\mathfrak{g}}\qd{\alpha}\frac{S_{\alpha w}}{S_{\v w}}\qd w^{2}\\
=\frac{1}{\cs{\du{\mathfrak{g}}}}\sum_{\alpha\in\mathfrak{g}}\qd{\alpha}\sum_{w}S_{\alpha w}\frac{S_{\v w}}{S_{\v\v}^{2}}=\frac{S_{\v\v}^{-2}}{\cs{\du{\mathfrak{g}}}}\sum_{\alpha\in\mathfrak{g}}\!\qd{\alpha}\delta_{\alpha\v}=\frac{1}{\cs{\du{\mathfrak{g}}}}\sum_{p}\qd p^{2}
\end{gather*}
\end{singlespace}

\noindent proving the assertion.
\end{proof}
The above results illustrate the inherent duality of $\fc$s: $\chg$
and $\du{\chg}$ determine each other, while their extents are, roughly
speaking, reciprocal. This duality means that any result about $\fc$s
holds simultaneously for $\mathfrak{g}$ and its dual $\du{\mathfrak{g}}$.
In particular, any result proven about classes gives a corresponding
result about blocks, and \emph{vice versa.} This seemingly trivial
observation turns out to be quite useful.
\begin{lem}
\label{lem:inclusion}If $\mathfrak{g}$ and $\mathfrak{h}$ are $\fc$s
such that $\mathfrak{h\!\subseteq\!\mathfrak{g}}$, then every $\mathfrak{h}$-class
is a union of $\mathfrak{g}$-classes, in particular $\du{\mathfrak{g}}\!\subseteq\!\du{\mathfrak{h}}$,
and every $\mathfrak{g}$-block is a union of $\mathfrak{h}$-blocks.
\end{lem}
\begin{proof}
If the primaries $p$ and $q$ belong to the same $\mathfrak{g}$-class,
i.e. if the restrictions to $\svera{\mathfrak{g}}$ of the irreps
$\ch p$ and $\ch q$ coincide, then for $\mathfrak{h}\!\subseteq\!\mathfrak{g}$
their restrictions to $\svera{\mathfrak{h}}$ coincide as well, showing
that every $\mathfrak{g}$-class is contained in a unique $\mathfrak{h}$-class,
hence each $\mathfrak{h}$-class is the union of the $\mathfrak{g}$-classes
that it contains. Since the $\mathfrak{g}$-class containing the vacuum
primary $\v$ is $\du{\mathfrak{g}}$, while the $\mathfrak{h}$-class
containing it is $\du{\mathfrak{h}}$, this gives $\du{\mathfrak{g}}\!\subseteq\!\du{\mathfrak{h}}$,
and consequently every $\mathfrak{g}$-block (i.e. $\du{\mathfrak{g}}$-class)
is a union of $\mathfrak{h}$-blocks ($\du{\mathfrak{h}}$-classes)
by the above argument.
\end{proof}
It follows from \prettyref{lem:inclusion} that every $\mathfrak{g}$-class
is a union of $\mathfrak{g}$-blocks precisely when $\mathfrak{g}\!\subseteq\!\du{\mathfrak{g}}$.
It turns out that such $\fc$s play a basic role in orbifold deconstruction
\cite{Bantay2018a,Bantay2018b}, hence they deserve a special name.
\begin{defn}
\label{def:local}An $\fc$ $\mathfrak{g}$ is local if $\chg\!\subseteq\!\du{\chg}$.\smallskip{}
\end{defn}
We shall encounter local $\fc$s in the sequel on several occasions.
A major feature of this notion explaining its special standing is
that, as a consequence of \prettyref{lem:localcrit} and a result
of Deligne \cite{Deligne1990}, the corresponding subalgebra $\svera{\mathfrak{g}}$
may be identified with the character ring of some finite group, hence
results from character theory \cite{Isaacs,Lux-Pahlings,Serre} go
over to local $\fc$s. This observation allows the generalization
of many group theoretic notions to arbitrary $\fc$s, and provides
a host of non-trivial conjectural results that seem to hold in full
generality. As an example, consider the following notion.
\begin{defn}
The central character of a class $\ccl\!\in\!\cl{\mathfrak{g}}$ is
the complex valued function $\map{\boldsymbol{\varpi}_{\ccl}}{\mathfrak{g}}{\mathbb{C}}$
assigning to $\alpha\!\in\!\mathfrak{g}$ the value
\begin{equation}
\cech{\ccl}{\alpha}\!=\!\frac{\cs{\du{\mathfrak{g}}}}{\cs{\ccl}}\frac{\alpha\!\left(\ccl\right)}{\qd{\alpha}}\label{eq:cechardef}
\end{equation}
\end{defn}
It is clear that the values of the central character are always algebraic
numbers. For local $\fc$s this notion gives back the corresponding
classical one from character theory, and by well known results \cite{Isaacs,Lux-Pahlings},
the values taken in that case are actually algebraic integers. Surprisingly,
this seems to be true for generic $\fc$s.
\begin{conjecture}
\label{conj:algint}$\cech{\ccl}{\alpha}$ is always an algebraic
integer.
\end{conjecture}
\begin{rem}
\noindent The truth of \prettyref{conj:algint} would imply that the
ratios
\[
\frac{\cs{\du{\mathfrak{g}}}}{\cs{\ccl}}\!=\!\cech{\ccl}{\v}
\]
are algebraic integers for every $\ccl\!\in\!\cl{\mathfrak{g}}$,
leading to the following (conjectural) analogue of Lagrange's theorem:
if $\mathfrak{g}$ and $\mathfrak{h}$ are $\fc$s and $\mathfrak{h\!\subseteq\!\mathfrak{g}}$,
then $\cs{\du{\mathfrak{h}}}$ divides $\cs{\du{\mathfrak{g}}}$,
i.e. their ratio is an algebraic integer.
\end{rem}
The inclusion relation makes the collection $\lat$ of $\fc$s partially
ordered, with maximal element the set of all primaries, and minimal
element the trivial $\fc$ $\left\{ \v\right\} $ consisting of the
vacuum primary solely. Because the intersection of two $\fc$s is
obviously an $\fc$ again, $\lat$ is actually a finite lattice \cite{Crawley1973,Gratzer2011}.
\begin{prop}
Given $\fc$s $\mathfrak{g}$ and $\mathfrak{h}$, their join $\mathfrak{g}\!\vee\!\mathfrak{h}$
(the smallest $\fc$ that contains both of them) is given by
\begin{equation}
\mathfrak{g}\!\vee\!\mathfrak{h}=\du{\left(\du{\mathfrak{g}}\cap\du{\mathfrak{h}}\right)}\label{eq:joindef}
\end{equation}
hence the map that sends each $\fc$ $\mathfrak{g}$ to $\du{\mathfrak{g}}$
is an isomorphism between the lattice $\lat$ and its dual.
\end{prop}
\begin{proof}
Since $\mathfrak{g,\mathfrak{h}\!\subseteq\!\mathfrak{g}\!\vee\!\mathfrak{h}}$
by definition, \prettyref{lem:inclusion} implies $\du{\left(\mathfrak{g}\!\vee\!\mathfrak{h}\right)}\!\subseteq\!\du{\mathfrak{g}},\du{\mathfrak{h}}$,
hence $\du{\left(\mathfrak{g}\!\vee\!\mathfrak{h}\right)}\!\subseteq\!\du{\mathfrak{g}}\!\cap\!\du{\mathfrak{h}}$,
that is $\du{\left(\du{\mathfrak{g}}\cap\du{\mathfrak{h}}\right)}\!\subseteq\!\mathfrak{g}\!\vee\!\mathfrak{h}$.
On the other hand, $\du{\mathfrak{g}}\!\cap\du{\mathfrak{h}}\!\subseteq\!\du{\mathfrak{g}},\du{\mathfrak{h}}$,
hence $\mathfrak{g},\mathfrak{h}\!\subseteq\!\du{\left(\du{\mathfrak{g}}\!\cap\du{\mathfrak{h}}\right)}$
again by \prettyref{lem:inclusion}, or in other words $\mathfrak{g}\!\vee\!\mathfrak{h}\!\subseteq\!\du{\left(\du{\mathfrak{g}}\cap\du{\mathfrak{h}}\right)}$,
proving the claim.
\end{proof}
\begin{thm}
\label{thm:modlat}The lattice $\lat$ of $\fc$s is modular (even
Arguesian), but usually not distributive.
\end{thm}
\begin{proof}
The map that assigns to each $\fc$ $\mathfrak{g}$ the collection
$\bl{\mathfrak{g}}$ of its $\dcl$s is clearly an injective embedding
of $\lat$ into the partition lattice of the set of all primaries,
hence one has to prove the assertion for the image of this homomorphism.
To prove modularity (or the stronger Arguesian property) of the latter,
all we have to show is that for any pair $\mathfrak{g,\mathfrak{h}}\!\in\!\lat$,
if $\mathfrak{b}_{1}\!\in\!\bl{\mathfrak{g}}$ and $\mathfrak{b}_{2}\!\in\!\bl{\mathfrak{h}}$
are blocks such that $\mathfrak{b}_{1}\!\cap\!\mathfrak{b}_{2}\!\not=\!\emptyset$,
then there is a block $\mathfrak{b}\!\in\!\bl{\mathfrak{g}\!\vee\!\mathfrak{h}}$
that contains both of them \cite{Jonsson1953}.   But \prettyref{lem:inclusion}
implies that in case $\mathfrak{g}\!\subseteq\!\mathfrak{g}\!\vee\!\mathfrak{h}$
there exists for each $\mathfrak{b}_{1}\!\in\!\bl{\mathfrak{g}}$
a block $\mathfrak{B}_{1}\!\in\!\bl{\mathfrak{g}\!\vee\!\mathfrak{h}}$
such that $\mathfrak{b}_{1}\!\subseteq\!\mathfrak{B}_{1}$, and a
similar argument shows that for $\mathfrak{b}_{2}\!\in\!\bl{\mathfrak{h}}$
there exists $\mathfrak{B}_{2}\!\in\!\bl{\mathfrak{g}\!\vee\!\mathfrak{h}}$
such that $\mathfrak{b}_{2}\!\subseteq\!\mathfrak{B}_{2}$. Since
$\mathfrak{b}_{1}\!\cap\!\mathfrak{b}_{2}\!\subseteq\!\mathfrak{B}_{1}\!\cap\!\mathfrak{B}_{2}$,
and two blocks are either equal or disjoint, we get that $\mathfrak{B}_{1}\!=\!\mathfrak{B}_{2}$
if $\mathfrak{b}_{1}\!\cap\!\mathfrak{b}_{2}\!\not=\!\emptyset$,
and obviously $\mathfrak{B}_{1}$ contains both $\mathfrak{b}_{1}$
and $\mathfrak{b}_{2}$. As to distributivity, it already fails for
a holomorphic $\mathbb{Z}_{2}$-orbifold (e.g. the $SO(16)$ Wess-Zumino
model at level $1$).
\end{proof}
\begin{rem}
A better understanding of the lattice theoretic properties of $\lat$
would be highly desirable. We just mention that, while $\lat$ is
modular according to the above, it is usually neither atomic nor complemented.
In particular, it is unclear whether $\lat$ admits a coordinatization
in the spirit of \cite{DayPickering,cont_geometries}. Another interesting
question, inspired by the results of \cite{Palfy1995}, is to find
extra identities satisfied by $\lat$.
\end{rem}
Going back to general properties of $\fc$s, note that, according
to \prettyref{lem:blcrit2}, restricting the indices of the fusion
matrices $\fm\!\left(\alpha\right)$ to the primaries belonging to
a given $\dcl$ $\mathfrak{\mathfrak{b}\!\in\!\bl{\mathfrak{g}}}$
results in non-negative integer matrices $\fm_{\mathfrak{b}}\!\left(\alpha\right)$
that form a representation $\bfm_{\mathfrak{b}}$ of the subalgebra
$\svera{\mathfrak{g}}$. As a consequence of Eq.\prettyref{eq:ver4},
for any primary $w$ belonging to the class $\ccl\!\in\!\cl{\mathfrak{g}}$
one has
\begin{equation}
\sum_{q\in\mathfrak{b}}\fm_{\mathfrak{b}}\!\left(\alpha\right)_{p}^{q}S_{wq}=\ch w\!\left(\alpha\right)S_{wp}=\ch{\ccl}\!\left(\alpha\right)S_{wp}\label{eq:bfmeq}
\end{equation}
for all $\alpha\!\in\!\mathfrak{g}$, reflecting the fact that $\bfm_{\mathfrak{b}}$
decomposes as a direct sum of the irreducible representations $\ch{\ccl}$.
\begin{defn}
\label{def:ovdef}~The overlap $\sp{\mathfrak{b}}{\ccl}$ of the
block $\mathfrak{b}\!\in\!\bl{\mathfrak{g}}$ and the class $\ccl\!\in\!\cl{\mathfrak{g}}$
is the multiplicity of the irrep $\ch{\ccl}$ in the irreducible decomposition
of the integral representation $\bfm_{\mathfrak{b}}$.
\end{defn}
\global\long\def\minors{S_{\mathfrak{b}\ccl}}%

\begin{lem}
\label{lem:ovrk}The overlap $\sp{\mathfrak{b}}{\ccl}$ equals the
rank of the minor $\minors$ of the modular $S$-matrix obtained by
restricting the row indices to $\mathfrak{b}\!\in\!\bl{\mathfrak{g}}$
and the column indices to $\ccl\!\in\!\cl{\mathfrak{g}}$.
\end{lem}
\begin{proof}
Since $\fm_{\mathfrak{b}}\!\left(\alpha\right)\!\minors\!=\!\ch{\ccl}\!\left(\alpha\right)\!\minors$
by Eq.\prettyref{eq:bfmeq}, the columns of $\minors$ span the invariant
subspace of $\bfm_{\mathfrak{b}}$ corresponding to the irrep $\ch{\ccl}$.
As the latter appears with multiplicity $\sp{\mathfrak{b}}{\ccl}$
in $\bfm_{\mathfrak{b}}$, we get the assertion.
\end{proof}
\begin{cor}
\label{cor:trivoverlap} $\sp{\mathfrak{b}}{\ccl}\!=\!1$ iff the
minor $\minors$ factorizes, i.e. there exist complex functions $\map{\xi}{\ccl}{\mathbb{C}}$
and $\map{\eta}{\mathfrak{b}}{\mathbb{C}}$ such that $S_{pq}\!=\!\xi\!\left(p\right)\!\eta\!\left(q\right)$
for $p\!\in\!\ccl$ and $q\!\in\!\mathfrak{b}$. In particular, $\sp{\mathfrak{g}}{\ccl}\!=\!1$
for every $\gcl$ $\ccl\!\in\!\cl{\mathfrak{g}}$, and $\sp{\mathfrak{b}}{\du{\mathfrak{g}}}\!=\!1$
for all $\mathfrak{b}\!\in\!\bl{\mathfrak{g}}$.
\end{cor}
\begin{proof}
Since rank $1$ matrices factorize, the first statement is a special
case of \prettyref{lem:ovrk}, and because $S_{\alpha p}\!=\!\alpha\!\left(\ccl\right)S_{\v p}$
for $p\!\in\!\ccl$ and $\alpha\!\in\!\mathfrak{g}$, this implies
at once $\sp{\mathfrak{g}}{\ccl}\!=\!1$, while $\sp{\mathfrak{b}}{\du{\mathfrak{g}}}\!=\!1$
follows from this by duality.
\end{proof}
\begin{lem}
\begin{singlespace}
\label{lem:clsize}For $\mathfrak{b}\!\in\!\bl{\mathfrak{g}}$ one
has
\begin{equation}
\sum_{\ccl\in\cl{\mathfrak{g}}}\sp{\mathfrak{b}}{\ccl}\!=\!\FA{\mathfrak{b}}\label{eq:blocksize}
\end{equation}
and for $\ccl\!\in\!\cl{\mathfrak{g}}$
\begin{equation}
\sum_{\mathfrak{b}\in\bl{\mathfrak{g}}}\sp{\mathfrak{b}}{\ccl}\!=\!\FA{\mathfrak{\ccl}}\label{eq:clsize}
\end{equation}
\end{singlespace}
\end{lem}
\begin{proof}
\begin{singlespace}
To prove Eq.\prettyref{eq:blocksize}, observe that
\[
\sum_{\ccl\in\cl{\mathfrak{g}}}\sp{\mathfrak{b}}{\ccl}\!=\!\sum_{\ccl\in\cl{\mathfrak{g}}}\sp{\mathfrak{b}}{\ccl}\dim\ch{\ccl}\!=\!\dim\bfm_{\mathfrak{b}}\!=\!\FA{\mathfrak{b}}
\]
\end{singlespace}

\noindent since $\dim\ch{\ccl}\!=\!1$. The second statement follows
from this by duality.
\end{proof}
\begin{lem}
\begin{singlespace}
\label{lem:overlapint}
\begin{equation}
\sp{\mathfrak{b}}{\ccl}=\sum_{p\in\mathfrak{b}}\sum_{q\in\ccl}\FA{S_{pq}}^{2}\label{eq:overlapdef}
\end{equation}
\end{singlespace}
\end{lem}
\begin{proof}
Since the irrep $\ch{\ccl}$ of the subalgebra $\svera{\mathfrak{g}}$
appears with multiplicity $\sp{\mathfrak{b}}{\ccl}$ in the irreducible
decomposition of $\bfm_{\mathfrak{b}}$, the matrix $\fm_{\mathfrak{b}}\!\left(\alpha\right)$
has $\sp{\mathfrak{b}}{\ccl}$ eigenvalues equal to $\!\ch{\ccl}\!\left(\alpha\right)$
for $\alpha\!\in\!\mathfrak{g}$, hence
\[
\sum_{\ccl\in\cl{\mathfrak{g}}}\sp{\mathfrak{b}}{\ccl}\,\alpha\!\left(\ccl\right)=\mathrm{Tr}(\fm_{\mathfrak{b}}\!\left(\alpha\right))=\sum_{p\in\mathfrak{b}}N_{\alpha p}^{p}
\]
i.e.
\begin{gather*}
\sp{\mathfrak{b}}{\ccl}\!=\!\frac{1}{\cs{\ccl}}\!\sum_{\alpha\in\chg}\!\overline{\alpha\!\left(\ccl\right)}\biggl\{\!\sum_{p\in\mathfrak{b}}\!N_{\alpha p}^{p}\!\biggr\}\!\!=\!\frac{1}{\cs{\ccl}}\!\sum_{\alpha\in\chg}\!\overline{\alpha\!\left(\ccl\right)}\Biggl\{\!\sum_{p\in\mathfrak{b}}\!\sum_{w}\!\frac{S_{\alpha w}S_{pw}\overline{S_{pw}}}{S_{\v w}}\!\Biggr\}\\
=\!\sum_{p\in\mathfrak{b}}\!\sum_{w}\!\FA{S_{pw}}^{2}\biggl\{\!\frac{1}{\cs{\ccl}}\!\sum_{\alpha\in\chg}\!\overline{\alpha\!\left(\ccl\right)}\frac{S_{\alpha w}}{S_{\v w}}\!\biggr\}\!=\!\sum_{p\in\mathfrak{b}}\!\sum_{w}\!\FA{S_{pw}}^{2}\delta_{\ccl}\!\left(w\right)\!=\!\sum_{p\in\mathfrak{b}}\sum_{q\in\ccl}\FA{S_{pq}}^{2}
\end{gather*}
using the orthogonality relation Eq.\prettyref{eq:ortho2}.
\end{proof}
\begin{cor}
\begin{singlespace}
$\sp{\mathfrak{b}}{\ccl}\!=\!0$ iff $S_{pq}\!=\!0$ for all $p\!\in\!\mathfrak{b}$
and $q\!\in\!\ccl$.
\end{singlespace}
\end{cor}
\begin{thm}[Reciprocity relation]
If the $\fc$s $\mathfrak{g},\mathfrak{h}\!\in\!\lat$ satisfy $\mathfrak{g}\!\subseteq\!\mathfrak{h}$,
then for every $\mathfrak{b}\!\in\!\bl{\mathfrak{h}}$ and $\ccl\!\in\!\cl{\mathfrak{g}}$
\begin{equation}
\sum_{{\textstyle {\mathfrak{b^{'}}\in\bl{\mathfrak{g}}\atop \mathfrak{b^{'}}\subseteq\mathfrak{b}}}}\sp{\mathfrak{b^{'}}}{\ccl}=\sum_{{\textstyle {\mathfrak{\ccl^{'}}\in\cl{\mathfrak{h}}\atop \mathfrak{\ccl^{'}}\subseteq\ccl}}}\sp{\mathfrak{b}}{\mathfrak{\ccl^{'}}}\label{eq:reciprocity}
\end{equation}
\end{thm}
\begin{proof}
This follows from \prettyref{lem:inclusion}, since
\begin{gather*}
\sum_{{\textstyle {\mathfrak{b^{'}}\in\bl{\mathfrak{g}}\atop \mathfrak{b^{'}}\subseteq\mathfrak{b}}}}\!\sp{\mathfrak{b^{'}}}{\ccl}\!=\!\sum_{{\textstyle {\mathfrak{b^{'}}\in\bl{\mathfrak{g}}\atop \mathfrak{b^{'}}\subseteq\mathfrak{b}}}}\sum_{p\in\mathfrak{b^{'}}}\Bigl(\sum_{q\in\ccl}\!\FA{S_{pq}}^{2}\Bigr)\!=\!\sum_{p\in\mathfrak{b}}\sum_{q\in\ccl}\!\FA{S_{pq}}^{2}\\
\!=\!\sum_{{\textstyle {\mathfrak{\ccl^{'}}\in\cl{\mathfrak{h}}\atop \mathfrak{\ccl^{'}}\subseteq\ccl}}}\sum_{q\in\mathfrak{\ccl^{'}}}\Bigl(\sum_{p\in\mathfrak{b}}\!\FA{S_{pq}}^{2}\Bigr)\!=\!\sum_{{\textstyle {\mathfrak{\ccl^{'}}\in\cl{\mathfrak{h}}\atop \mathfrak{\ccl^{'}}\subseteq\ccl}}}\!\sp{\mathfrak{b}}{\mathfrak{\ccl^{'}}}
\end{gather*}
as a consequence of Eq.\prettyref{eq:overlapdef}.
\end{proof}
\begin{cor}
If $\mathfrak{g}\!\subseteq\!\mathfrak{h}$, then the number of $\mathfrak{h}$-classes
contained in $\du{\mathfrak{g}}$ equals the number of $\mathfrak{g}$-blocks
contained in $\mathfrak{h}$.
\end{cor}
\begin{proof}
Apply \prettyref{cor:trivoverlap} and Eq.\prettyref{eq:reciprocity}
with $\ccl\!=\!\du{\mathfrak{g}}$ and $\mathfrak{b}\!=\!\mathfrak{h}$.
\end{proof}
\begin{prop}
\begin{singlespace}
\label{prop:overlapbound}For all $\ccl\!\in\!\cl{\mathfrak{g}}$
and $\mathfrak{b}\!\in\!\bl{\mathfrak{g}}$ one has
\begin{equation}
\sp{\mathfrak{b}}{\ccl}\!\leq\!\min\left(\frac{\cs{\du{\mathfrak{g}}}}{\cs{\ccl}},\frac{\cs{\mathfrak{g}}}{\cs{\mathfrak{b}}}\right)\label{eq:ovbound}
\end{equation}
\end{singlespace}
\end{prop}
\begin{proof}
\begin{singlespace}
Since
\begin{gather*}
\sp{\mathfrak{b}}{\ccl}\!=\!\frac{1}{\cs{\ccl}}\sum_{\alpha\in\mathfrak{g}}\overline{\alpha\!\left(\ccl\right)}\mathrm{Tr}~\fm_{\mathfrak{b}}\!\left(\alpha\right)
\end{gather*}
while $\mathrm{Tr}~\fm_{\mathfrak{b}}\!\left(\alpha\right)$ is non-negative
and $\FA{\alpha\!\left(\ccl\right)\!}\!\leq\!\qd{\alpha}$, one has
\[
\sp{\mathfrak{b}}{\ccl}\!\leq\!\frac{1}{\cs{\ccl}}\!\sum_{\alpha\in\mathfrak{g}}\!\FA{\overline{\alpha\!\left(\ccl\right)}\mathrm{Tr}~\fm_{\mathfrak{b}}\!\left(\alpha\right)\!}\!\leq\!\frac{1}{\cs{\ccl}}\!\sum_{\alpha\in\mathfrak{g}}\!\qd{\alpha}\mathrm{Tr}~\fm_{\mathfrak{b}}\!\left(\alpha\right)\!=\!\frac{\cs{\du{\mathfrak{g}}}}{\cs{\ccl}}
\]
\end{singlespace}

\noindent by the triangle inequality, taking into account \prettyref{cor:trivoverlap}.
Reversing the role of $\mathfrak{g}$ and $\du{\mathfrak{g}}$ completes
the proof by duality.
\end{proof}
\begin{cor}
\label{cor:extbound}$\cs{\ccl}\!\leq\!\cs{\du{\mathfrak{g}}}$ for
$\ccl\!\in\!\cl{\mathfrak{g}}$, and $\cs{\mathfrak{b}}\!\leq\!\cs{\mathfrak{g}}$
for $\mathfrak{b}\!\in\!\bl{\mathfrak{g}}$.
\end{cor}
\begin{proof}
Indeed, $\cs{\ccl}\!=\!\cs{\ccl}\!\sp{\mathfrak{g}}{\ccl}\!\leq\!\cs{\du{\mathfrak{g}}}$
by \prettyref{cor:trivoverlap} and Eq.\prettyref{eq:ovbound}. The
second statement follows by duality.
\end{proof}
The properties of those classes $\ccl\!\in\!\cl{\mathfrak{g}}$ that
saturate the bound $\cs{\ccl}\!\leq\!\cs{\du{\mathfrak{g}}}$ will
be discussed in the next section.
\begin{lem}
\label{lem:sizebound1}$\FA{\ccl}\!\leq\!\cs{\mathfrak{g}}$ for $\ccl\!\in\!\cl{\mathfrak{g}}$,
and $\FA{\mathfrak{b}}\!\leq\!\cs{\du{\mathfrak{g}}}$ for $\mathfrak{b}\!\in\!\bl{\mathfrak{g}}$.
\end{lem}
\begin{proof}
\begin{singlespace}
The first inequality follows from Eqs.\prettyref{eq:ovbound} and
\prettyref{eq:clsize}, since
\[
\FA{\ccl}\!=\!\sum_{\mathfrak{b}\in\bl{\mathfrak{g}}}\sp{\mathfrak{b}}{\ccl}\!\leq\!\sum_{\mathfrak{b}\in\bl{\mathfrak{g}}}\frac{\cs{\mathfrak{g}}}{\cs{\mathfrak{b}}}\!=\!\cs{\mathfrak{g}}
\]
\end{singlespace}

\noindent taking into account Eq.\prettyref{eq:extsumrule}, and the
second one follows by duality.
\end{proof}

\section{The center\label{sec:The-center}}
\begin{defn}
\label{centerdef}The center $\zent{\mathfrak{g}}$ of the $\fc$
$\mathfrak{g}\!\in\!\lat$ is the collection of those $\mathfrak{g}$-$\gcl$es
$\zcl\!\in\!\cl{\mathfrak{g}}$ for which $\cs{\zcl}\!=\!\cs{\du{\chg}}$.
\end{defn}
Clearly $\du{\mathfrak{g}}\!\in\!\zent{\mathfrak{g}}$, hence the
center is never empty: we'll call the elements of $\zent{\mathfrak{g}}$
central $\gcl$es. It follows from \prettyref{prop:overlapbound}
that for a central $\gcl$ $\zcl\!\in\!\zent{\mathfrak{g}}$ one has
$\sp{\mathfrak{b}}{\zcl}\!\leq\!1$ for any $\dcl$ $\mathfrak{b}\!\in\!\bl{\mathfrak{g}}$,
hence $\FA{\zcl}\!\leq\!\FA{\mathfrak{\du g}}$ by Eq.\prettyref{eq:blocksize}.
\begin{lem}
\label{lem:centclchar}$\zcl\!\in\!\zent{\mathfrak{g}}$ iff $\FA{\alpha\!\left(\zcl\right)\!}\!=\!\qd{\alpha}$
for all $\alpha\!\in\!\chg$, i.e. the central character of $\zcl$
has unit modulus.
\end{lem}
\begin{proof}
\begin{singlespace}
If $\FA{\alpha\!\left(\zcl\right)\!}\!=\!\qd{\alpha}$ for all $\alpha\!\in\!\chg$,
then
\[
\cs{\zcl}=\sum_{\alpha\in\mathfrak{g}}\FA{\alpha\!\left(\zcl\right)\!}^{2}=\sum_{\alpha\in\mathfrak{g}}\qd{\alpha}^{2}=\cs{\du{\chg}}
\]
 by Eq.\prettyref{eq:ortho2}. Conversely, $\cs{\zcl}\!=\!\cs{\du{\chg}}$
implies
\[
\sum_{\alpha\in\mathfrak{g}}\left(\qd{\alpha}^{2}\!-\!\FA{\alpha\!\left(\zcl\right)\!}^{2}\right)\!=\!0
\]
\end{singlespace}

\noindent Since $\FA{\alpha\!\left(\zcl\right)\!}=\FA{\ch{\zcl}\!\left(\alpha\right)\!}\leq\qd{\alpha}$,
all terms of the sum on the left-hand side are non-negative, hence
they should all vanish.
\end{proof}
\begin{cor}
If $\mathfrak{g},\mathfrak{h}\!\in\!\lat$ are $\fc$s such that $\mathfrak{h}\!\subseteq\!\mathfrak{g}$,
then any class $\ccl\!\in\!\cl{\mathfrak{h}}$ containing a central
class $\zcl\!\in\!\zent{\mathfrak{g}}$ is itself central.
\end{cor}
\begin{proof}
Since $\zcl\!\subseteq\!\ccl$, we have $\alpha\!\left(\ccl\right)\!=\!\ch{\ccl}\!\left(\alpha\right)\!=\!\ch{\zcl}\!\left(\alpha\right)\!=\!\cech{\zcl}{\alpha}\!\qd{\alpha}$
for $\alpha\!\in\!\mathfrak{h}$, i.e. $\FA{\alpha\!\left(\zcl\right)\!}\!=\!\qd{\alpha}$,
proving the assertion according to \prettyref{lem:centclchar}.
\end{proof}
\begin{defn}
\label{def:Abel}An $\fc$ $\mathfrak{g}\!\in\!\lat$ is Abelian if
all its classes are central.
\end{defn}
\begin{prop}
An $\fc$ $\mathfrak{g}\!\in\!\lat$ is Abelian iff $\qd{\alpha}\!=\!1$
for all $\alpha\!\in\!\mathfrak{g}$.
\end{prop}
\begin{proof}
\noindent If all classes $\ccl\!\in\!\cl{\mathfrak{g}}$ are central,
then by Eq.\prettyref{eq:extsumrule}
\[
\cs{\du{\mathfrak{g}}}=\sum_{\ccl\in\cl{\mathfrak{g}}}\frac{\cs{\du{\mathfrak{g}}}}{\cs{\ccl}}=\FA{\cl{\mathfrak{g}}\!}=\FA{\mathfrak{g}}
\]
hence by Eq.\prettyref{eq:spread}
\[
0=\cs{\du{\mathfrak{g}}}-\FA{\mathfrak{g}}=\sum_{\alpha\in\mathfrak{g}}\left(\qd{\alpha}^{2}-1\right)
\]
Since $\qd{\alpha}\!\geq\!1$, we get the only if part. On the other
hand, if $\qd{\alpha}\!=\!1$ for all $\alpha\!\in\!\mathfrak{g}$,
then $\cs{\du{\mathfrak{g}}}\!=\!\FA{\mathfrak{g}}$, hence by Eqs.\prettyref{eq:extsumrule}
and \prettyref{eq:classno}
\[
\sum_{\ccl\in\cl{\mathfrak{g}}}\!\left(\frac{1}{\cs{\ccl}}\!-\!\frac{1}{\cs{\du{\mathfrak{g}}}}\right)\!=\!0
\]
Since all terms of the sum are non-negative by \prettyref{cor:extbound},
they must all vanish.
\end{proof}
\begin{rem}
In the language of 2D CFT, Abelian $\fc$s are groups of simple currents.
\end{rem}
\begin{lem}
\label{lem:cechprod}If $\alpha,\!\beta,\!\gamma\!\in\!\mathfrak{g}$
are such that $N_{\alpha\beta}^{\gamma}\!>\!0$ and $\zcl\!\in\!\zent{\mathfrak{g}}$
is a central class, then $\cech{\zcl}{\gamma}\!=\!\cech{\zcl}{\alpha}\!\cech{\zcl}{\beta}$. 
\end{lem}
\begin{proof}
\begin{singlespace}
By definition of the central character
\[
\cech{\zcl}{\alpha}\!\cech{\zcl}{\beta}\!\qd{\alpha}\qd{\beta}\!=\!\alpha\!\left(\zcl\right)\!\beta\!\left(\zcl\right)\!=\!\sum_{\gamma\in\mathfrak{g}}N_{\alpha\beta}^{\gamma}\gamma\!\left(\zcl\right)\!=\!\sum_{\gamma\in\mathfrak{g}}N_{\alpha\beta}^{\gamma}\cech{\zcl}{\gamma}\!\qd{\gamma}
\]
which is equivalent to
\[
\sum_{\gamma\in\mathfrak{g}}N_{\alpha\beta}^{\gamma}\qd{\gamma}\left(1-\frac{\cech{\zcl}{\gamma}}{\cech{\zcl}{\alpha}\!\cech{\zcl}{\beta}}\right)=0
\]
\end{singlespace}

\noindent Since the real part of a complex number cannot exceed its
modulus, all terms of the sum should equal $0$.
\end{proof}
\begin{lem}
\label{lem:centralchar}If $p$ and $q$ are primaries such that $N_{\alpha p}^{q}\!>\!0$
for some $\alpha\!\in\!\mathfrak{g}$, then $p$ belongs to the central
$\gcl$ $\zcl\!\in\!\zent{\mathfrak{g}}$ iff
\begin{equation}
\frac{\om{\alpha}\om p}{\om q}\!=\!\cech{\zcl}{\alpha}\label{eq:centralchar}
\end{equation}
\end{lem}
\begin{proof}
This follows at once from \prettyref{lem:centclchar} and \prettyref{lem:omch}.
\end{proof}
\begin{cor}
\begin{singlespace}
\label{cor:trivclass} 
\[
\du{\mathfrak{g}}\!=\!\set p{\om q\!=\!\om{\alpha}\!\om p\mathcal{\textrm{ \emph{if} }}N_{\alpha p}^{q}\!>\!0\mathcal{\textrm{ \emph{for} }}\alpha\!\in\!\mathfrak{g}}
\]
\end{singlespace}
\end{cor}
\begin{prop}
\label{prop:zentprod}If $\ccl$ is a $\mathfrak{g}$-$\gcl$ and
$\mathtt{z}\!\in\!\zent{\mathfrak{g}}$, then there exist unique $\mathfrak{g}$-$\gcl$es
$\zcl^{\pm1}\ccl\!\in\!\cl{\mathfrak{g}}$ such that
\begin{equation}
\alpha\!\left(\zcl^{\pm1}\ccl\right)=\cech{\zcl}{\alpha}^{\pm1}\!\alpha\!\left(\ccl\right)\label{eq:zentproddef}
\end{equation}
for all $\alpha\!\in\!\mathfrak{g}$; in particular, $\du{\mathfrak{g}}\ccl\!=\!\ccl$.
Moreover,
\begin{equation}
\cs{\zcl^{\pm1}\ccl}\!=\!\cs{\ccl}\label{eq:centprodext}
\end{equation}
and
\begin{equation}
\cech{\zcl\ccl}{\alpha}\!=\!\cech{\zcl}{\alpha}\cech{\ccl}{\alpha}\label{eq:cechprod}
\end{equation}
\end{prop}
\begin{proof}
\begin{singlespace}
According to \prettyref{lem:cechprod}, one has
\begin{gather*}
\sum_{\gamma\in\mathfrak{g}}\!N_{\alpha\beta}^{\gamma}\cech{\zcl}{\gamma}^{\pm1}\!\gamma\!\left(\ccl\right)\!=\!\left\{ \cech{\zcl}{\alpha}\!\cech{\zcl}{\beta}\right\} ^{\pm1}\!\sum_{\gamma\in\mathfrak{g}}\!N_{\alpha\beta}^{\gamma}\gamma\!\left(\ccl\right)\!\\
=\!\left\{ \cech{\zcl}{\alpha}^{\pm1}\!\alpha\!\left(\ccl\right)\right\} \!\left\{ \cech{\zcl}{\beta}^{\pm1}\!\beta\!\left(\ccl\right)\right\} 
\end{gather*}
\end{singlespace}

\noindent which means that the product $\boldsymbol{\varpi}_{\zcl}^{\pm1}\ch{\ccl}$
is itself an irrep of the algebra $\vera_{\mathfrak{g}}$, hence it
is equal to the irrep corresponding to some well defined $\mathfrak{g}$-$\gcl$,
namely $\zcl^{\pm1}\ccl$. That $\du{\mathfrak{g}}\ccl\!=\!\ccl$
follows from $\cech{\du{\mathfrak{g}}}{\alpha}\!=\!1$ for $\alpha\!\in\!\mathfrak{g}$.
Finally, Eq.\eqref{eq:ortho2} gives
\[
\cs{\zcl^{\pm1}\ccl}\!=\!\sum_{\alpha\in\mathfrak{g}}\FA{\alpha\!\left(\zcl^{\pm1}\ccl\right)\!}^{2}\!=\!\sum_{\alpha\in\mathfrak{g}}\FA{\alpha\!\left(\ccl\right)\!}^{2}\!=\!\cs{\ccl}
\]
proving Eq.\prettyref{eq:centprodext}, leading to Eq.\prettyref{eq:cechprod}
when combined with Eq.\prettyref{eq:zentproddef}.
\end{proof}
Note the following generalization of the product rule \prettyref{thm: Product rule}.
\begin{thm}
\label{thm:genprodrule}If $p$ belongs to the $\gcl$ $\ccl\!\in\!\cl{\mathfrak{g}}$
and $q$ belongs to the central $\gcl$ $\zcl\!\in\!\zent{\mathfrak{g}}$,
then $N_{pq}^{r}\!>\!0$ implies $r\!\in\!\zcl\ccl$.
\end{thm}
\begin{proof}
According to Eqs.\prettyref{eq:classcharfun}, \prettyref{eq:cechprod}
and \prettyref{eq:ortho2}
\begin{gather*}
\sum_{r\in\zcl\ccl}N_{pq}^{r}\qd r\!=\!\sum_{r}\!\clchar{\zcl\ccl}rN_{pq}^{r}\qd r\!=\!\frac{1}{\cs{\zcl\ccl}}\sum_{\alpha\in\mathfrak{g}}\overline{\alpha\!\left(\zcl\ccl\right)}\sum_{r}\frac{S_{\alpha r}}{S_{\v r}}N_{pq}^{r}\qd r\!=\!\\
\frac{1}{\cs{\ccl}}\sum_{\alpha\in\mathfrak{g}}\overline{\alpha\!\left(\ccl\right)}\cech{\zcl}{\alpha}^{-1}\sum_{r}N_{pq}^{r}\frac{S_{\alpha r}}{S_{\v\v}}\!=\!\frac{1}{\cs{\ccl}}\sum_{\alpha\in\mathfrak{g}}\overline{\alpha\!\left(\ccl\right)}\cech{\zcl}{\alpha}^{-1}\frac{S_{\alpha p}}{S_{\alpha\v}}\frac{S_{\alpha q}}{S_{\v\v}}\\
\!=\!\frac{1}{\cs{\ccl}}\sum_{\alpha\in\mathfrak{g}}\overline{\alpha\!\left(\ccl\right)}\cech{\zcl}{\alpha}^{-1}\frac{\alpha\!\left(\ccl\right)S_{\v p}}{S_{\alpha\v}}\frac{\alpha\!\left(\zcl\right)S_{\v q}}{S_{\v\v}}\!=\!\frac{\qd p\qd q}{\cs{\ccl}}\sum_{\alpha\in\mathfrak{g}}\overline{\alpha\!\left(\ccl\right)}\alpha\!\left(\ccl\right)\!=\!\qd p\qd q
\end{gather*}
hence
\[
\sum_{r\not\in\zcl\ccl}N_{pq}^{r}\qd r\!=\!0
\]
Since all terms on the left-hand side are non-negative, $N_{pq}^{r}\!=\!0$
for $r\not\in\zcl\ccl$.
\end{proof}
\begin{prop}
\label{prop:centergroup} If $\zcl_{1},\!\zcl_{2}\!\in\!\zent{\mathfrak{g}}$
are central classes, then $\zcl_{1}\zcl_{2}\!=\!\zcl_{2}\zcl_{1}$
is also central, and $(\zcl_{1}\zcl_{2})\ccl\!=\!\zcl_{1}(\zcl_{2}\ccl)$
for all $\ccl\!\in\!\cl{\mathfrak{g}}$.
\end{prop}
\begin{proof}
If $\zcl_{1},\!\zcl_{2}\!\in\!\zent{\mathfrak{g}}$, then $\cs{\zcl_{1}\zcl_{2}}\!=\!\cs{\zcl_{2}}\!=\!\cs{\du{\mathfrak{g}}}$
by Eq.\eqref{eq:centprodext}, proving that $\zcl_{1}\zcl_{2}\!\in\!\zent{\mathfrak{g}}$.
By Eq.\prettyref{eq:cechprod}, $\cech{\zcl_{1}\zcl_{2}}{\alpha}\!=\!\cech{\zcl_{1}}{\alpha}\!\cech{\zcl_{2}}{\alpha}\!=\!\cech{\zcl_{2}}{\alpha}\!\cech{\zcl_{1}}{\alpha}\!=\!\cech{\zcl_{2}\zcl_{1}}{\alpha}$
for $\alpha\!\in\!\mathfrak{g}$, and because central characters of
different classes differ from each other, this shows that $\zcl_{1}\zcl_{2}\!=\!\zcl_{2}\zcl_{1}$.
Finally, again by Eq.\prettyref{eq:cechprod}
\begin{gather*}
\cech{(\zcl_{1}\zcl_{2})\ccl}{\alpha}\!=\!\cech{\zcl_{1}\zcl_{2}}{\alpha}\!\cech{\ccl}{\alpha}\!=\!\cech{\zcl_{1}}{\alpha}\!\cech{\zcl_{2}}{\alpha}\!\cech{\ccl}{\alpha}\!=\!\cech{\zcl_{1}}{\alpha}\!\cech{\zcl_{2}\ccl}{\alpha}
\end{gather*}

\noindent for $\ccl\!\in\!\cl{\mathfrak{g}}$, hence $(\zcl_{1}\zcl_{2})\ccl\!=\!\zcl_{1}(\zcl_{2}\ccl)$,
finishing the proof.
\end{proof}
\begin{thm}
The center $\zent{\mathfrak{g}}$ of an $\fc$ $\mathfrak{g}\!\in\!\lat$
is an Abelian group that permutes the $\mathfrak{g}$-classes.
\end{thm}
\begin{proof}
Defining $\zcl_{1}\zcl_{2}$ as the product of the central $\gcl$es
$\zcl_{1},\!\zcl_{2}\!\in\!\zent{\mathfrak{g}}$, \prettyref{prop:centergroup}
implies that it is commutative and associative. Since $\du{\mathfrak{g}}\zcl\!=\!\zcl$
for every $\zcl\!\in\!\zent{\mathfrak{g}}$, $\du{\mathfrak{g}}$
is the identity element of this product, and the $\gcl$ $\zcl\inv\du{\mathfrak{g}}\!\in\!\zent{\mathfrak{g}}$
is clearly the inverse of $\zcl$, since $\zcl\left(\zcl\inv\du{\mathfrak{g}}\right)\!=\!\du{\mathfrak{g}}$,
proving that $\zent{\mathfrak{g}}$ is indeed an Abelian group. Finally,
again by \prettyref{prop:centergroup} the maps $\ccl\mapsto\zcl\ccl$
define a permutation action of $\zent{\mathfrak{g}}$ on the set $\cl{\mathfrak{g}}$
of $\mathfrak{g}$-classes.
\end{proof}

\section{Central quotients and extensions\label{sec:Central-quotients-and}}

\global\long\def\idch{\boldsymbol{\mathfrak{1}}}%
\global\long\def\usub#1{\mathbf{\boldsymbol{\cup}}#1}%
\global\long\def\zquot#1#2{\mathfrak{#1}/\!#2}%
\global\long\def\mzq#1#2{#1^{[#2]}}%

\global\long\def\extclass#1{#1\ccl}%

\begin{prop}
\begin{singlespace}
\label{prop:zentext}For an $\fc$ $\mathfrak{g}$ and a subgroup
$Z\!<\!\zent{\mathfrak{g}}$ of its center,
\begin{equation}
\zquot{\mathfrak{g}}Z=\set{\alpha\!\in\!\mathfrak{g}}{\alpha\!\left(\zcl\right)\!\!=\!\qd{\alpha}\textrm{ \emph{for} \emph{all} }\zcl\!\in\!Z}\label{eq:zentextdef}
\end{equation}
is again an $\fc$, the central quotient of $\mathfrak{g}$ by $Z$,
with dual
\begin{equation}
\du{\left(\zquot{\mathfrak{g}}Z\right)}=\bigcup_{\zcl\in Z}\zcl\label{eq:zentextdual}
\end{equation}
If $\mathfrak{h}\!\in\!\lat$ is such that $\zquot{\mathfrak{g}}Z\!\!\subseteq\!\!\mathfrak{h}\!\subseteq\!\mathfrak{g}$,
then $\mathfrak{h}\!=\!\zquot{\mathfrak{g}}H$ for some subgroup $H\!<\!Z$.
\end{singlespace}
\end{prop}
\begin{proof}
To simplify notation, let $\mathfrak{\mathfrak{g}_{\idch}}$ denote
$\zquot{\mathfrak{g}}Z$, and $\usub Z$ the union of the classes
in $Z$. It follows from \prettyref{lem:cechprod} that $N_{\alpha\beta}^{\gamma}\!>\!0$
for $\alpha,\beta\!\in\!\mathfrak{g}_{\idch}\!=\!\set{\alpha\!\in\!\mathfrak{g}}{\cech{\zcl}{\alpha}\!=\!1\textrm{ for all }\zcl\!\in\!Z}$
implies $\gamma\!\in\!\mathfrak{g}_{\idch}$, hence $\mathfrak{g}_{\idch}\!\in\!\lat$.
Clearly, $\usub Z\!\subseteq\!\du{\mathfrak{g}_{\idch}}$ since $\alpha\!\left(\zcl\right)\!=\!\qd{\alpha}$
for $\zcl\!\in\!Z$ and $\alpha\!\in\!\mathfrak{g}_{\idch}$, while
\begin{gather*}
\cs{\du{\mathfrak{g}_{\idch}}}\!=\!\sum_{\alpha\in\mathfrak{g}_{\idch}}\!\qd{\alpha}^{2}\!=\!\sum_{\alpha\in\mathfrak{g}}\!\frac{1}{\FA Z}\sum_{\zcl\in Z}\!\cech{\zcl}{\alpha}\!\qd{\alpha}^{2}\!=\!\frac{1}{\FA Z}\sum_{\zcl\in Z}\sum_{\alpha\in\mathfrak{g}}\alpha\!\left(\zcl\right)\!\qd{\alpha}\!=\!\frac{\cs{\du{\mathfrak{g}}}}{\FA Z}
\end{gather*}

\noindent by Eqs.\prettyref{eq:spread} and \prettyref{eq:ortho2}.
Since $\cs{\zcl}\!=\!\cs{\du{\mathfrak{g}}}$ for all $\zcl\!\in\!\zent{\mathfrak{g}}$,
one gets
\begin{gather*}
\sum_{p\in\du{\mathfrak{g}_{\idch}}\setminus\usub Z}\!S_{\v p}^{2}\!=\!\sum_{p\in\du{\mathfrak{g}_{\idch}}}\!S_{\v p}^{2}\!-\!\sum_{p\in\usub Z}\!S_{\v p}^{2}\!=\!\frac{1}{\cs{\du{\mathfrak{g}_{\idch}}}}\!-\!\sum_{\zcl\in Z}\!\frac{1}{\cs{\zcl}}\!=\!\frac{\FA Z}{\cs{\du{\mathfrak{g}}}}\!-\!\FA Z\frac{1}{\cs{\du{\mathfrak{g}}}}\!=\!0
\end{gather*}
which implies that $\du{\mathfrak{g}_{\idch}}\!\setminus\!\usub Z$
is void, since $S_{\v p}^{2}\!>\!0$ for all $p$.

Finally, if $\mathfrak{g}_{\idch}\!\subseteq\!\mathfrak{h}\!\subseteq\!\mathfrak{g}$
then $\du{\mathfrak{g}}\!\subseteq\!\du{\mathfrak{h}}\!\subseteq\!\du{\mathfrak{g}_{\idch}}$
by \prettyref{lem:inclusion}, hence $\du{\mathfrak{h}}$ is a union
of $\mathfrak{g}$-classes contained in $\du{\mathfrak{g}_{\idch}}\!=\!\usub Z$;
consequently, $\du{\mathfrak{h}}\!=\!\usub H$ for some subset $H\!\subseteq\!Z$,
and because $\du{\mathfrak{h}}$ is an $\fc$, $H$ is actually a
subgroup of $Z$ such that $\du{\left(\zquot{\mathfrak{g}}H\right)}\!=\!\usub H\!=\!\du{\mathfrak{h}}$
by Eq.\prettyref{eq:zentextdual}, i.e. $\mathfrak{h}\!=\!\zquot{\mathfrak{g}}H$.
\end{proof}
It follows from the above result that there is an order reversing
one-to-one correspondence between central quotients of $\mathfrak{g}\!\in\!\lat$
and subgroups of its center $\zent{\mathfrak{g}}$. The usefulness
of central quotients rests on the following result.
\begin{thm}
\label{thm:zentextstruct}For a subgroup $Z\!<\!\zent{\mathfrak{g}}$
of the center of $\mathfrak{g}\!\in\!\lat$, let $\dg Z\!=\!\mathtt{Hom}\!\left(Z,\mathbb{C}^{\times}\right)$
denote its character group (Pontryagin dual), and let $\mathfrak{g}_{\xi}\!=\!\set{\alpha\!\in\!\mathfrak{g}}{\cech{\zcl}{\alpha}\!=\!\xi\!\left(\zcl\right)\textrm{ \emph{for} }\zcl\!\in\!Z}$
for $\xi\!\in\!\dg Z$. Then
\end{thm}
\begin{enumerate}[label=\arabic{enumi})]
\item \label{enu:zentfc} each $\mathfrak{g}_{\xi}$ is a block of $\zquot{\mathfrak{g}}Z$,
 of cardinality
\[
\FA{\mathfrak{g}_{\xi}}\!=\!\frac{1}{\FA Z}\sum_{\zcl\in Z}\overline{\xi\!\left(\zcl\right)}\!\FA{\fix{\zcl}\!}
\]
where $\fix{\zcl}\!=\!\set{\ccl\!\in\!\cl{\mathfrak{g}}\!\!}{\zcl\ccl\!=\!\ccl}$
denotes the set of fixed points of the central class $\zcl\!\in\!\zent{\mathfrak{g}}$
in its action on $\cl{\mathfrak{g}}$;
\item \label{enu:zentbl}\label{enu:gxiprod}$\cs{\mathfrak{g}_{\xi}}\!=\!\FA Z\!\cs{\mathfrak{g}}$
for $\xi\!\in\!\dg Z$, i.e. each $\mathfrak{g}_{\xi}$ belongs to
the center of $\du{\left(\zquot gZ\right)}$, and $\mathfrak{g}_{\xi}\mathfrak{g}_{\eta}\!=\!\mathfrak{g}_{\xi\eta}$
for $\xi,\eta\!\in\!\dg Z$, hence $\du Z\!=\!\set{\mathfrak{g}_{\xi}}{\xi\!\in\!\dg Z}$
is a subgroup of the center of $\du{\left(\zquot{\mathfrak{g}}Z\right)}$
 isomorphic to $Z$, and $\zquot{\du{\left(\zquot{\mathfrak{g}}{\mathit{Z}}\right)}\!}{\du Z}\!=\!\du{\mathfrak{g}}$;
\item \label{enu:zentclass}each $\zquot{\mathfrak{g}}Z$-class is of a
union $\extclass Z\!=\!\bigcup\limits _{\zcl\in Z}\!\zcl\ccl$ for
some class $\ccl\!\in\!\cl{\mathfrak{g}}$, with $\cs{\ccl}=\left[Z\!:\!Z_{\ccl}\right]\cs{Z\ccl}$
and
\begin{equation}
\sp{\mathfrak{g}_{\xi}}{\extclass Z}\!=\!\begin{cases}
1 & \textrm{if }\xi\!\left(\zcl\right)\!=\!1\textrm{ for all }\zcl\!\in\!Z_{\ccl};\\
0 & \textrm{otherwise}
\end{cases}\label{eq:ovgxi}
\end{equation}
where $Z_{\ccl}\!=\!\set{\zcl\!\in\!Z}{\zcl\ccl\!=\!\ccl}$ denotes
the stabilizer of $\ccl$.
\end{enumerate}
\begin{proof}
\prettyref{lem:cechprod} implies that if $N_{\alpha\beta}^{\gamma}\!>\!0$
for $\alpha\!\in\!\mathfrak{g}_{\xi}$ and $\beta\!\in\!\mathfrak{g}_{\eta}$,
then $\gamma\!\in\!\mathfrak{g}_{\xi\eta}$. In particular, if $\idch$
denotes the principal character of $Z$ (the identity of $\dg Z$),
then $\mathfrak{g}_{\idch}$ is an $\fc$, and it follows from \prettyref{lem:blcrit2}
that $\mathfrak{g}_{\xi}\!\in\!\bl{\mathfrak{g}_{\idch}}$ for each
$\xi\!\in\!\dg Z$. Since clearly $\mathfrak{g}_{\idch}\!=\!\zquot gZ$,
we get the first assertion. Next, notice that for $\zcl\!\in\!\zent{\mathfrak{g}}$
one has by Eq.\prettyref{eq:ortho2}
\begin{gather*}
\FA{\fix{\zcl}\!}\!=\!\sum_{\ccl\in\cl{\mathfrak{g}}}\!\delta_{\ccl,\zcl\ccl}\!=\!\sum_{\ccl\in\cl{\mathfrak{g}}}\!\frac{1}{\cs{\ccl}}\sum_{\alpha\in\mathfrak{g}}\!\alpha\!\left(\zcl\ccl\right)\overline{\alpha\!\left(\ccl\right)}\!=\!\\
\sum_{\ccl\in\cl{\mathfrak{g}}}\!\frac{1}{\cs{\ccl}}\sum_{\alpha\in\mathfrak{g}}\!\cech{\zcl}{\alpha}\!\FA{\alpha\!\left(\ccl\right)\!}^{2}\!=\!\sum_{\alpha\in\mathfrak{g}}\!\cech{\zcl}{\alpha}\!\sum_{\ccl\in\cl{\mathfrak{g}}}\!\frac{\FA{\alpha\!\left(\ccl\right)\!}^{2}}{\cs{\ccl}}\!=\!\sum_{\alpha\in\mathfrak{g}}\!\cech{\zcl}{\alpha}
\end{gather*}
leading to
\[
\FA{\mathfrak{g}_{\xi}}\!=\!\sum_{\alpha\in\mathfrak{g}}\!\frac{1}{\FA Z}\sum_{\zcl\in Z}\!\overline{\xi\!\left(\zcl\right)}\cech{\zcl}{\alpha}\!=\!\frac{1}{\FA Z}\sum_{\zcl\in Z}\!\overline{\xi\!\left(\zcl\right)}\!\sum_{\alpha\in\mathfrak{g}}\!\cech{\zcl}{\alpha}\!=\!\frac{1}{\FA Z}\sum_{\zcl\in Z}\overline{\xi\!\left(\zcl\right)}\!\FA{\fix{\zcl}\!}
\]

\begin{singlespace}
To prove \prettyref{enu:gxiprod}, note that
\begin{gather*}
\frac{\cs{\mathfrak{g}}}{\cs{\mathfrak{g}_{\xi}}}\!=\!\frac{\sum_{\alpha\in\mathfrak{g}_{\xi}}\!S_{\v\alpha}^{2}}{\sum_{\alpha\in\mathfrak{g}}\!S_{\v\alpha}^{2}}\!=\!\frac{1}{\cs{\du{\mathfrak{g}}}}\sum_{\alpha\in\mathfrak{g}}\!\qd{\alpha}^{2}\frac{1}{\FA Z}\sum_{\zcl\in Z}\overline{\xi\!\left(\zcl\right)}\cech{\zcl}{\alpha}\!\\
=\!\frac{1}{\FA Z\cs{\du{\mathfrak{g}}}}\sum_{\zcl\in Z}\overline{\xi\!\left(\zcl\right)}\sum_{\alpha\in\mathfrak{g}}\alpha\!\left(\zcl\right)\!\qd{\alpha}\!=\!\frac{1}{\FA Z}\sum_{\zcl\in Z}\overline{\xi\!\left(\zcl\right)}\delta_{\zcl,\du{\mathfrak{g}}}\!=\!\frac{1}{\FA Z}
\end{gather*}
\end{singlespace}

\noindent is independent of $\xi\!\in\!\dg Z$, where we have used
\prettyref{lem:spreadrecip}. It follows that for all $\xi\!\in\!\dg Z$
one has $\cs{\mathfrak{g}_{\xi}}\!=\!\cs{\mathfrak{g}_{\idch}}$,
which is tantamount to $\mathfrak{g}_{\xi}\!\in\!\zent{\du{\mathfrak{g}_{\idch}}}$.
That $\mathfrak{g}_{\xi}\mathfrak{g}_{\eta}\!=\!\mathfrak{g}_{\xi\eta}$
for $\xi,\eta\!\in\!\dg Z$ can be seen as follows: $N_{\alpha\beta}^{\gamma}\!>\!0$
with $\alpha\!\in\!\mathfrak{g}_{\xi}$ and $\beta\!\in\!\mathfrak{g}_{\eta}$
implies that $\gamma\!\in\!\mathfrak{g}_{\xi}\mathfrak{g}_{\eta}$
by the generalized product rule \prettyref{thm:genprodrule}, and
$\gamma\!\in\!\mathfrak{g}_{\xi\eta}$ by \prettyref{lem:cechprod}.
But this means that the map $\xi\mapsto\mathfrak{g}_{\xi}$ sets up
an isomorphism $\dg Z\!\cong\!\du Z$, and because $\dg Z\!\cong\!Z$
by Pontryagin duality, we get that $Z\!\cong\!\du Z$. Finally, it
is clear that
\[
\usub{\du Z}\!=\!\bigcup_{\xi\in\dg Z}\!\mathfrak{g}_{\xi}\!=\!\mathfrak{g}
\]
hence $\du{\mathfrak{g}}\!=\!\du{\left(\usub{\du Z}\right)}\!=\!\zquot{\du{\left(\zquot{\mathfrak{g}}{\mathit{Z}}\right)}\!}{\du Z}$
according to Eq.\prettyref{eq:zentextdual}.

As to \prettyref{enu:zentclass}, notice that (since $\mathfrak{g}_{\idch}\!\subseteq\!\mathfrak{g}$)
each $\mathfrak{g}_{\idch}$-class $\mathfrak{C}\!\in\!\cl{\mathfrak{g}_{\idch}}$
is a union of $\mathfrak{g}$-classes by \prettyref{lem:inclusion},
hence there exists some $\ccl\!\in\!\cl{\mathfrak{g}}$ contained
in $\mathfrak{C}$. But for $\zcl\!\in\!Z$ the restrictions of $\ch{\ccl}$
and $\ch{\zcl\ccl}$ to $\mathfrak{g}_{\idch}$ coincide, consequently
one has $\extclass Z\!\subseteq\!\mathfrak{C}$. To prove that this
containment is actually an equality, observe that one has
\begin{gather*}
\sum_{\alpha\in\mathfrak{g}_{\xi}}\!\sum_{p\in\extclass Z}\!\FA{S_{\alpha p}}^{2}\!=\!\frac{1}{\FA{Z_{\ccl}}}\sum_{\zcl\in Z}\!\sum_{p\in\zcl\ccl}\sum_{\alpha\in\mathfrak{g}_{\xi}}\!\FA{S_{\alpha p}}^{2}\!=\!\frac{1}{\FA{Z_{\ccl}}}\sum_{\zcl\in Z}\!\sum_{\alpha\in\mathfrak{g}_{\xi}}\!\FA{\alpha\!\left(\zcl\ccl\right)\!}^{2}\!\sum_{p\in\zcl\ccl}\!S_{\v p}^{2}\\
\!=\!\frac{\left[Z\!:\!Z_{\ccl}\right]}{\cs{\ccl}}\sum_{\alpha\in\mathfrak{g}_{\xi}}\!\FA{\alpha\!\left(\ccl\right)\!}^{2}\!=\!\frac{\left[Z\!:\!Z_{\ccl}\right]}{\cs{\ccl}}\sum_{\alpha\in\mathfrak{g}}\frac{1}{\FA Z}\sum_{\zcl\in Z}\overline{\xi\!\left(\zcl\right)}\cech{\zcl}{\alpha}\!\FA{\alpha\!\left(\ccl\right)\!}^{2}\\
\!=\!\frac{1}{\FA{Z_{\ccl}}}\sum_{\zcl\in Z}\frac{\overline{\xi\!\left(\zcl\right)}}{\cs{\ccl}}\sum_{\alpha\in\mathfrak{g}}\!\alpha\!\left(\zcl\ccl\right)\overline{\alpha\!\left(\ccl\right)}=\!\frac{1}{\FA{Z_{\ccl}}}\sum_{\zcl\in Z_{\ccl}}\overline{\xi\!\left(\zcl\right)}
\end{gather*}
 for $\xi\!\in\!\dg Z$, and in particular
\[
\sum_{\alpha\in\mathfrak{g}_{\idch}}\sum_{p\in\mathfrak{C}\setminus\extclass Z}\!\FA{S_{\alpha p}}^{2}\!=\!\sum_{\alpha\in\mathfrak{g}_{\idch}}\sum_{p\in\mathfrak{C}}\!\FA{S_{\alpha p}}^{2}\!-\!\sum_{\alpha\in\mathfrak{g}_{\idch}}\sum_{p\in\extclass Z}\!\FA{S_{\alpha p}}^{2}\!=\!\sp{\mathfrak{g}_{\idch}}{\mathfrak{C}}\!-\!1\!=\!0
\]

\noindent because $\sp{\mathfrak{g}_{\idch}}{\mathfrak{C}}\!=\!1$
for all $\mathfrak{C}\!\in\!\cl{\mathfrak{g}_{\idch}}$ according
to \prettyref{cor:trivoverlap}; since $\mathfrak{g}_{\idch}$ contains
the vacuum $\v$, this can only happen if $\mathfrak{C}\setminus\extclass Z$
is empty, proving that indeed $\mathfrak{C}\!=\!\extclass Z$. Taking
this into account, one has
\[
\frac{1}{\cs{Z\ccl}}\!=\!\frac{1}{\FA{Z_{\ccl}}}\sum_{\zcl\in Z}\sum_{p\in\zcl\ccl}\!S_{\v p}^{2}\!=\!\frac{1}{\FA{Z_{\ccl}}}\sum_{\zcl\in Z}\!\cs{\zcl\ccl}\!=\!\frac{\left[Z\!:\!Z_{\ccl}\right]}{\cs{\ccl}}
\]
Finally,
\[
\sp{\mathfrak{g}_{\xi}}{\extclass Z}\!=\!\sum_{\alpha\in\mathfrak{g}_{\xi}}\sum_{p\in\extclass Z}\FA{S_{\alpha p}}^{2}\!=\!\begin{cases}
1 & \textrm{if }Z_{\ccl}\!<\!\ker\xi;\\
0 & \textrm{otherwise }
\end{cases}
\]

\noindent according to the above, proving the last assertion.
\end{proof}
Given an $\fc$ $\mathfrak{g}$, it is natural to ask whether it is
a central quotient of another $\fc$. This leads to the following
notion.
\begin{defn}
Let $\mathfrak{g}\!\in\!\lat$ denote an $\fc$ and $A$ an Abelian
group. An $A$-extension of $\mathfrak{g}$ is an $\fc$ $\mathfrak{h}\!\in\!\lat$
such that $\zquot{\mathfrak{h}}Z\!=\!\mathfrak{g}$ for some central
subgroup $Z\!<\!\zent{\mathfrak{h}}$ isomorphic to $A$.
\end{defn}
\begin{lem}
For an Abelian group $A$ and $\mathfrak{g}\!\in\!\lat$, the different
$A$-extensions of $\mathfrak{g}$ are in one-to-one correspondence
with subgroups of $\zent{\du{\mathfrak{g}}}$ isomorphic to $A$.
\end{lem}
\begin{proof}
Suppose that $\mathfrak{h}\!\in\!\lat$ is an $A$-extension of $\mathfrak{g}$,
i.e. $\mathfrak{g}\!=\!\zquot hZ$ for some subgroup $Z\!<\!\zent{\mathfrak{h}}$
isomorphic to $A$. By \prettyref{enu:gxiprod} of \prettyref{thm:zentextstruct}
$\du Z\!=\!\set{\mathfrak{h}_{\xi}}{\xi\!\in\!\dg Z}$ is a subgroup
of $\zent{\du{\mathfrak{g}}}$ isomorphic to $Z$, hence to $A$.
Conversely, for any subgroup $Z\!<\!\zent{\du{\mathfrak{g}}}$ isomorphic
to $A$, $\mathfrak{h}\!=\!\du{\left(\zquot{\du{\mathfrak{g}}}Z\right)}$
is an $\fc$, and by \prettyref{enu:gxiprod} of \prettyref{thm:zentextstruct}
one has $\zquot h{\du Z}\!=\!\zquot{\du{\left(\zquot{\du{\mathfrak{g}}}{\mathit{Z}}\right)}}{\du Z\!=\!\mathfrak{g}}$
for some subgroup $\du Z\!<\!\zent{\mathfrak{h}}$ isomorphic to $Z$,
hence to $A$ as well, i.e. $\mathfrak{h}$ is an $A$-extension of
$\mathfrak{g}$.
\end{proof}
\global\long\def\cov#1{#1^{\intercal}}%

\begin{cor}
Every $\fc$ $\mathfrak{g}\!\in\!\lat$ has a maximal central extension,
the dual of the maximal central quotient of $\du{\mathfrak{g}}$.
\end{cor}
\begin{defn}
\label{def:nilp}An $\fc$ $\mathfrak{g}\!\in\!\lat$ is nilpotent
if it can be obtained from the trivial $\fc$ by a sequence of central
extensions.
\end{defn}
The rationale of this terminology is that if $\mathfrak{g}\!\in\!\lat$
is local, hence the associated algebra $\svera{\mathfrak{g}}$ is
isomorphic to the character ring of some finite group $G$, the $\fc$
$\mathfrak{g}$ is nilpotent according to the above definition precisely
when $G$ is nilpotent.
\begin{lem}
\label{lem:nilext}If $\mathfrak{g}\!\in\!\lat$ is nilpotent, then
$\cs{\du{\mathfrak{g}}}\!\in\!\mathbb{Z}$.
\end{lem}
\begin{proof}
According to \prettyref{thm:zentextstruct}, if $\mathfrak{g}\!\in\!\lat$
is a central extension of $\mathfrak{h}\!\in\!\lat$, then $\cs{\du{\mathfrak{g}}}$
is an integer multiple of $\cs{\du{\mathfrak{h}}}$. The claim follows
by induction.
\end{proof}
We'll see in \prettyref{cor:intspread} that $\cs{\du{\mathfrak{g}}}\!\in\!\mathbb{Z}$
implies that the quantum dimension of any element of $\mathfrak{g}$
is either an integer or the square root of an integer. That the latter
possibility can occur is exemplified by the maximal $\fc$ of the
Ising model (the minimal Virasoro model of central charge $\tfrac{1}{2}$),
which is nilpotent while having a primary of dimension $\sqrt{2}$.
We conjecture that many results about (finite) nilpotent groups carry
over to nilpotent $\fc$s, like the following property, which is known
to be equivalent to nilpotency for finite groups.
\begin{conjecture}
If $\mathfrak{g}$ is nilpotent and $d$ is an integer dividing $\cs{\du{\mathfrak{g}}}$,
then there exists an $\fc$ $\mathfrak{h}\!\subseteq\!\mathfrak{g}$
such that $\cs{\du{\mathfrak{h}}}\!=\!d$.
\end{conjecture}

\section{The Galois action\label{sec:The-Galois-action}}

\global\long\def\eps#1#2{\boldsymbol{\upepsilon}_{#1}(#2)}%

Let us briefly recall the basics of the Galois action in RCFT \cite{Boere1991,Coste1994,Bantay2003b}.
It is known that, denoting by $N$ the least common multiple of the
denominators of the conformal weights $\cw p$, the field obtained
by adjoining to the rationals $\mathbb{Q}$ the quantum dimensions
$\qd p$ and the exponentiated conformal weights $\om p$ is the cyclotomic
field $\cyc N$ of conductor $N$, generated by a primitive root of
unity $\zeta_{N}\!=\!\exi[][2]N$. The Galois group of $\cyc N$ is
isomorphic to the group $\zn N$ of prime residues mod $N$, with
each $\ell\!\in\!\zn N$ corresponding to a Galois transformation
$\gal{\ell}$ mapping $\zeta_{N}$ to $\zeta_{N}^{\ell}$ (while leaving
all rationals fixed). Since the conformal weights of primaries are
rational numbers, one has $\gal{\ell}\circ\omega\!=\!\omega^{\ell}$
for the exponentiated conformal weights.

Because fusion matrices have rational integer matrix elements, the
irreducible representations of the Verlinde algebra $\vera$ (and
of its subalgebras) are permuted between themselves by the Galois
transformations $\gal{\ell}$. In other words, for each $\ell\!\in\!\zn N$
there exists a permutation $\uppi\!\left(\ell\right)\!:\!p\!\mapsto\!\galpi{\ell}p$
of the primaries such that
\begin{equation}
\gal{\ell}\!\circ\!\ch p=\ch{\galpi{\ell}p}\label{eq:galact}
\end{equation}
and the mapping $\uppi\!:\!\ell\!\mapsto\!\uppi\!\left(\ell\right)$
is clearly a homomorphism. Furthermore,
\begin{equation}
\gal{\ell}\!\left(\qd p\right)\!=\!\eps{\ell}p\frac{\qd{\galpi{\ell}p}}{\qd{\galpi{\ell}\v}}\label{eq:galact3}
\end{equation}
where $\eps{\ell}p=\pm1$ depending on the sign of $\gal{\ell}\!\left(\qd p\right)$.
Note that it follows from Eqs.\prettyref{eq:ineq} and \prettyref{eq:galact}
that $\FA{\gal{\ell}\!\left(\qd p\right)\!}\!=\!\FA{\ch{\galpi{\ell}\v}\!\left(p\right)\!}\!\leq\!\qd p$,
i.e. $\qd{\galpi{\ell}p}\!\leq\!\qd p\qd{\galpi{\ell}\v}$. Finally,
as a consequence of the fact that the modular representation has as
kernel a congruence subgroup of level $N$, one has for all primaries
$p$ and any $\ell\!\in\!\zn N$

\noindent 
\begin{equation}
\cw{\galpi{\ell}p}-\cw{\galpi{\ell}\v}\!\in\!\ell^{2}\cw p\!+\!\mathbb{Z}\label{eq:omgalpi}
\end{equation}

\global\long\def\mint#1{\boldsymbol{\Theta}_{#1}}%

\global\long\def\munt{\pd{\mint{\ell}}}%

\begin{lem}
\label{lem:intdef}For $\ell\!\in\!\zn N$ let $\mint{\ell}\!=\!\set p{\gal{\ell}\!\left(\qd p\right)\!=\!\eps{\ell}p\qd p}$.
Then $p,q\!\in\!\mint{\ell}$ and $N_{pq}^{r}\!>\!0$ implies $r\!\in\!\mint{\ell}$
and $\eps{\ell}r\!=\!\eps{\ell}p\eps{\ell}q$. As a consequence, both
$\mint{\ell}$ and $\munt\!=\!\set p{\gal{\ell}\!\left(\qd p\right)\!=\!\qd p}$
are $\fc$s.
\end{lem}
\begin{proof}
Applying $\gal{\ell}$ to both sides of ${\displaystyle \sum\limits _{r}}\!N_{pq}^{r}\qd r\!=\!\qd p\qd q$
gives
\[
\eps{\ell}p\eps{\ell}q\sum_{r}\!N_{pq}^{r}\gal{\ell}\!\left(\qd r\right)\!=\!\eps{\ell}p\eps{\ell}q\gal{\ell}\!\left(\qd p\qd q\right)\!=\!\qd p\qd q\!=\!\sum_{r}\!N_{pq}^{r}\qd r
\]
for $p,q\!\in\!\mint{\ell}$. After rearrangement and taking real
parts, one gets
\[
\sum_{r}\!N_{pq}^{r}\,\mathtt{Re}\Bigl\{\qd r-\eps{\ell}p\eps{\ell}q\gal{\ell}\!\left(\qd r\right)\Bigr\}=0
\]
Because $\FA{\gal{\ell}\!\left(\qd r\right)\!}\!\leq\!\qd r$, and
the real part of a complex number cannot exceed its modulus, we conclude
that all terms on the left-hand side are non-negative, hence they
should all vanish, that is $N_{pq}^{r}\!=\!0$ unless $\gal{\ell}\!\left(\qd r\right)\!=\!\eps{\ell}p\eps{\ell}q\qd r\!=\!\eps{\ell}r\qd r$.
Since $\munt\!=\!\set{p\!\in\!\mint{\ell}}{\eps{\ell}p\!=\!1}$, the
assertion follows.
\end{proof}
\begin{rem}
Note that, as a consequence of $\eps{\ell}p=\pm1$, the $\fc$ $\mint{\ell}$
either coincides with $\munt$ or is a $\mathbb{Z}_{2}$-extension
of it.
\end{rem}
\begin{cor}
\label{cor:intfc}Both $\pd{\mint{}}\!=\!\set p{\qd p\!\in\!\mathbb{Z}}$
and $\mint{}\!=\!\set p{\qd p^{2}\!\in\!\mathbb{Z}}$ are $\fc$s.
\end{cor}
\begin{proof}
Because quantum dimensions are algebraic integers, one has $p\!\in\!\pd{\mint{}}$
iff $\qd p\!\in\!\mathbb{Q}$, i.e. $\gal{\ell}\!\left(\qd p\right)\!=\!\qd p$
for all $\ell\!\in\!\zn N$, hence
\[
\pd{\mint{}}=\bigcap_{\ell\in\zn N}\!\pd{\mint{\ell}}
\]

\noindent This implies at once $\pd{\mint{}}\!\in\!\lat$ by \prettyref{lem:intdef}.
A similar argument works for $\mint{}$, exploiting the fact that
$\eps{\ell}p\!=\!\pm1$.
\end{proof}
From now on, we shall consider a fixed $\fc$ $\mathfrak{g}\!\in\!\lat$.
\begin{prop}
\noindent \label{prop:galact}The Galois permutations $\galpi{\uppi\!\left(\ell\right)}$
map $\mathfrak{g}$-$\gcl$es to $\mathfrak{g}$-$\gcl$es, i.e. $\galpi{\ell}\ccl\!=\!\set{\galpi{\ell}p}{p\!\in\!\ccl}\!\in\!\cl{\mathfrak{g}}$
for every $\gcl$ $\ccl\!\in\!\cl{\mathfrak{g}}$, in such a way that
$\galpi{\ell}\!\left(\zcl\ccl\right)\!=\!\zcl^{\ell}\!\left(\galpi{\ell}\ccl\right)$
for $\zcl\!\in\!\zent{\mathfrak{g}}$. Moreover, overlaps are left
invariant, $\sp{\mathfrak{b}}{\galpi{\ell}\ccl}\!=\!\sp{\mathfrak{b}}{\ccl}$
for all $\mathfrak{b}\!\in\!\bl{\mathfrak{g}}$, while $\cs{\galpi{\ell}\ccl}=\gal{\ell}\!\left(\cs{\ccl}\right)$.
\end{prop}
\begin{proof}
\begin{singlespace}
That $\galpi{\ell}\ccl\!\in\!\cl{\mathfrak{g}}$ follows at once from
Eq.\eqref{eq:galact}, while for $\!\zcl\!\in\!\zent{\mathfrak{g}}$
\begin{gather*}
\alpha\!\left(\galpi{\ell}\!\left(\zcl\ccl\right)\right)\!=\!\gal{\ell}\!\left(\alpha\!\left(\zcl\ccl\right)\right)\!=\!\gal{\ell}\!\left(\cech{\zcl}{\alpha}\!\alpha\!\left(\ccl\right)\right)\!=\!\\
\!=\!\cech{\zcl}{\alpha}^{\ell}\gal{\ell}\!\left(\alpha\!\left(\ccl\right)\right)\!=\!\cech{\zcl^{\ell}}{\alpha}\alpha\!\left(\galpi{\ell}\ccl\right)\!=\!\alpha\!\left(\zcl^{\ell}\!\left(\galpi{\ell}\ccl\right)\right)
\end{gather*}
by Eqs.\prettyref{eq:galact} and \prettyref{eq:zentproddef}, since
$\cech{\zcl}{\alpha}$ is a root of unity. As to the rest, 
\end{singlespace}

\noindent remember that the overlap $\sp{\mathfrak{b}}{\ccl}$ is
the multiplicity of the irrep $\ch{\ccl}$ in the irreducible decomposition
of the representation $\bfm_{\mathfrak{b}}$ associated to the block
$\mathfrak{b}$. Since $\bfm_{\mathfrak{b}}$ is integral, it equals
all its Galois conjugates, hence it contains the irreps $\ch{\galpi{\ell}\ccl}\!=\!\gal{\ell}\!\circ\!\ch{\ccl}$
and $\ch{\ccl}$ with the same multiplicity, i.e. $\sp{\mathfrak{b}}{\galpi{\ell}\ccl}\!=\!\sp{\mathfrak{b}}{\ccl}$.
Finally,
\[
\gal{\ell}\!\left(\cs{\ccl}\right)\!=\!\sum_{\alpha\in\mathfrak{g}}\!\gal{\ell}\!\left(\FA{\alpha\!\left(\ccl\right)\!}^{2}\right)\!=\!\sum_{\alpha\in\mathfrak{g}}\!\FA{\alpha\!\left(\galpi{\ell}\ccl\right)\!}^{2}\!=\!\cs{\galpi{\ell}\ccl}
\]

\noindent by Eq.\prettyref{eq:ortho2}, proving the assertion.
\end{proof}
\begin{cor}
\noindent \label{cor:galactbl}The Galois permutations $\galpi{\uppi\!\left(\ell\right)}$
map $\mathfrak{g}$-$\dcl$s to $\mathfrak{g}$-$\dcl$s, i.e. $\galpi{\ell}\mathfrak{b}\!=\!\set{\galpi{\ell}p}{p\!\in\!\mathfrak{b}}\!\in\!\bl{\mathfrak{g}}$
for any $\dcl$ $\mathfrak{b}\!\in\!\bl{\mathfrak{g}}$, in such a
way that $\cs{\galpi{\ell}\mathfrak{b}}=\gal{\ell}\!\left(\cs{\mathfrak{b}}\right)$
and $\sp{\galpi{\ell}\mathfrak{b}}{\ccl}\!=\sp{\mathfrak{b}}{\ccl}$
for all $\ccl\!\in\!\cl{\mathfrak{g}}$. As a consequence, every $\fc$
$\mathfrak{g}\!\in\!\lat$ is self-conjugate, i.e.the charge conjugate
$\overline{\alpha}$ of any primary $\alpha\!\in\!\mathfrak{g}$ also
belongs to $\mathfrak{g}$.
\end{cor}
\begin{proof}
The first claim follows from \prettyref{prop:galact} by duality.
As to the second, note that $\gal{\textrm{-}1}$ is complex conjugation,
hence $\galpi{\uppi\!\left(\textrm{-}1\right)}$ is charge conjugation.
Because $\galpi{\uppi\!\left(\textrm{-}1\right)}$ leaves the vacuum
$\v\!\in\!\mathfrak{g}$ invariant, it should map $\mathfrak{g}$
onto itself.
\end{proof}
\global\long\def\galg#1#2{#2#1}%

\begin{lem}
\label{lem:galdef}$\mathfrak{g}\!\subseteq\!\munt$ iff $\galg{\mathfrak{\du g}}{\ell}\!=\!\du{\mathfrak{g}}$.
\end{lem}
\begin{proof}
Since $\qd{\alpha}\!=\!\ch{\du{\mathfrak{g}}}\!\left(\alpha\right)$,
one has $\ch{\galpi{\ell}\du{\mathfrak{g}}}\!\left(\alpha\right)\!=\!\left(\gal{\ell}\!\circ\!\ch{\du{\mathfrak{g}}}\right)\!\left(\alpha\right)\!=\!\gal{\ell}\!\left(\qd{\alpha}\right)$
for $\alpha\!\in\!\mathfrak{g}$, hence $\galg{\mathfrak{\du g}}{\ell}\!=\!\du{\mathfrak{g}}$
exactly when $\gal{\ell}\!\left(\qd{\alpha}\right)\!=\!\qd{\alpha}$
for all $\alpha\!\in\!\mathfrak{g}$.
\end{proof}
\begin{lem}
\label{lem:galpi}If\emph{ }$\mathfrak{g}\!\subseteq\!\munt$ and
$\mathfrak{b}\!\in\!\bl{\mathfrak{g}}$, then the ratio $\dfrac{\gal{\ell}\left(\qd p\right)}{\qd p}$
is independent of $p\!\in\!\mathfrak{b}$.
\end{lem}
\begin{proof}
Let $\mathfrak{D}$ denote the vector whose components are the quantum
dimensions $\qd p$ for $p\!\in\!\mathfrak{b}$. Eq.\prettyref{eq:ver4}
implies that $\fm_{\mathfrak{b}}\!\left(\alpha\right)\!\mathfrak{D}\!=\!\qd{\alpha}\mathfrak{D}$
for all $\alpha\!\in\!\mathfrak{g}$, and applying $\gal{\ell}$ to
both sides gives $\fm_{\mathfrak{b}}\!\left(\alpha\right)\!\gal{\ell}\!\left(\mathfrak{D}\right)\!=\!\qd{\alpha}\gal{\ell}\!\left(\mathfrak{D}\right)$,
taking into account that the $\fm_{\mathfrak{b}}\!\left(\alpha\right)$
are integer matrices and $\gal{\ell}\!\left(\qd{\alpha}\right)\!=\!\qd{\alpha}$.
But this means that both $\mathfrak{D}$ and $\gal{\ell}\!\left(\mathfrak{D}\right)$
belong to the common eigenspace of the matrices $\fm_{\mathfrak{b}}\!\left(\alpha\right)$
corresponding to the irrep $\!\ch{\du{\mathfrak{g}}}$ of $\svera{\mathfrak{g}}$,
and because $\sp{\mathfrak{b}}{\du{\mathfrak{g}}}\!=\!1$ by \prettyref{cor:trivoverlap},
this eigenspace has dimension $1$, hence $\gal{\ell}\!\left(\mathfrak{D}\right)$
and $\mathfrak{D}$ are proportional to each other.
\end{proof}
It follows from \prettyref{cor:intfc} that $\pd{\intlat}\!=\!\set{\mathfrak{g}\!\subseteq\!\lat}{\mathfrak{g}\!\subseteq\!\pd{\mint{}}}$
is a sub\-lattice of $\lat$ consisting of those $\fc$s all of whose
elements have integer quantum dimension. Such $\fc$s have special
properties, as exemplified by the following result.
\begin{lem}
If $\mathfrak{g}\!\in\!\pd{\intlat}$, then for every block $\mathfrak{b}\!\in\!\bl{\mathfrak{g}}$
there exists an algebraic integer $\qd{\mathfrak{b}}$ such that the
quantum dimension of the primaries contained in $\mathfrak{b}$ are
rational multiples of $\qd{\mathfrak{b}}$.
\end{lem}
\begin{proof}
Using the notation $\qd{\mathfrak{b}}\!=\!\min\!\set{\qd p}{p\!\in\!\mathfrak{b}}$,
one has
\[
\gal{\ell}\!\left(\frac{\qd p}{\qd{\mathfrak{b}}}\right)\!=\!\frac{\qd p}{\qd{\mathfrak{b}}}
\]
for all $\ell\!\in\!\zn N$ by \prettyref{lem:galpi}. Since $\qd{\mathfrak{b}}$
is an algebraic integer, and an algebraic number fixed by all $\gal{\ell}$
is rational, the result follows.
\end{proof}
\begin{prop}
\label{prop:intspread}The following statements are equivalent:
\end{prop}
\begin{enumerate}[label=\alph*)]
\item $\mathfrak{g}\!\subseteq\!\mint{\ell}$;
\item $\gal{\ell}\!\left(\qd{\alpha}^{2}\right)\!=\!\qd{\alpha}^{2}$ for
all $\alpha\!\in\!\mathfrak{g}$;
\item $\gal{\ell}\!\left(\cs{\du{\mathfrak{g}}}\right)\!=\!\cs{\du{\mathfrak{g}}}$;
\item $\ell\du{\mathfrak{g}}\!\in\!\zent{\mathfrak{g}}$.
\end{enumerate}
\begin{proof}
For one thing, $\mathfrak{g}\!\subseteq\!\mint{\ell}$, i.e. $\gal{\ell}\!\left(\qd{\alpha}\right)\!=\!\eps{\ell}{\alpha}\qd{\alpha}\!=\!\pm\qd{\alpha}$
for all $\alpha\!\in\!\mathfrak{g}$ implies $\gal{\ell}\!\left(\qd{\alpha}^{2}\right)\!=\!\qd{\alpha}^{2}$,
which in turn implies $\gal{\ell}\!\left(\cs{\du{\mathfrak{g}}}\right)\!=\!\cs{\du{\mathfrak{g}}}$
by Eq.\prettyref{eq:spread}. But $\cs{\du{\mathfrak{g}}}\!=\!\gal{\ell}\!\left(\cs{\du{\mathfrak{g}}}\right)\!=\!\cs{\du{\ell\mathfrak{g}}}$
gives at once $\ell\du{\mathfrak{g}}\!\in\!\zent{\mathfrak{g}}$ and
\[
\cech{\ell\du{\mathfrak{g}}}{\alpha}\qd{\alpha}\!=\!\ch{\ell\du{\mathfrak{g}}}\!\left(\alpha\right)\!=\!\left(\gal{\ell}\!\circ\!\ch{\du{\mathfrak{g}}}\right)\!\left(\alpha\right)\!=\!\gal{\ell}\!\left(\qd{\alpha}\right)\!=\!\eps{\ell}{\alpha}\frac{\qd{\galpi{\ell}\alpha}}{\qd{\galpi{\ell}\v}}
\]
for $\alpha\!\in\!\mathfrak{g}$, as a consequence of Eq.\prettyref{eq:galact3}.
Since quantum dimensions are positive numbers while $\cech{\ell\du{\mathfrak{g}}}{\alpha}$
has unit modulus, one concludes that $\cech{\ell\du{\mathfrak{g}}}{\alpha}\!=\!\eps{\ell}{\alpha}$
and $\mathfrak{g}\!\subseteq\!\mint{\ell}$, completing the proof.
\end{proof}
\begin{cor}
\label{cor:intspread}$\cs{\du{\mathfrak{g}}}\!\in\!\mathbb{Z}$ iff
$\mathfrak{g}\!\subseteq\!\mint{}$.
\end{cor}
\begin{proof}
By \prettyref{prop:intspread}, $\cs{\du{\mathfrak{g}}}\!\in\!\mathbb{Z}$
iff $\mathfrak{g}\!\subseteq\!\mint{\ell}$ for all $\ell\!\in\!\zn N$,
i.e. $\mathfrak{g}\!\subseteq\!\mint{}$.
\end{proof}
By the above, $\intlat\!=\!\set{\mathfrak{g}\!\in\!\lat\!}{\!\cs{\du{\mathfrak{g}}}\!\in\!\mathbb{Z}}$
is a sublattice of $\lat$ containing $\pd{\intlat}$. Note that,
according to \prettyref{lem:nilext}, every nilpotent $\fc$ belongs
to $\intlat$, but the converse need not be true.
\begin{lem}
\begin{singlespace}
\label{lem:dimratio}If\emph{ }$\mathfrak{g}\!\subseteq\!\mint{\ell}$
and $\mathfrak{b}\!\in\!\bl{\mathfrak{g}}$, then for all $p\!\in\!\mathfrak{b}$
\begin{equation}
\frac{\qd{\galpi{\ell}p}}{\qd p}=\sqrt{\frac{\cs{\mathfrak{b}}}{\cs{\galpi{\ell}\mathfrak{b}}}}\label{eq:dimratio}
\end{equation}
\end{singlespace}
\end{lem}
\begin{proof}
\begin{singlespace}
\global\long\def\gp{\mathfrak{g}_{{\scriptscriptstyle +}}}%
If\emph{ }$\mathfrak{g}\!\subseteq\!\munt$, then it follows from
Eq.\prettyref{eq:galact3} and \prettyref{lem:galpi} that
\[
\frac{S_{\v\ell p}^{2}}{S_{\v p}^{2}}\!=\!\left(\frac{\qd{\ell p}}{\qd p}\right)^{2}\!=\!\left(\frac{\gal{\ell}\!\left(\qd p\right)}{\qd p}\right)^{2}\eps{\ell}p^{2}\qd{\ell\v}^{2}\!=\!A
\]
is independent of $p\!\in\!\mathfrak{b}$, hence
\[
\cs{\mathfrak{b}}\!=\!\sum_{p\in\mathfrak{b}}\frac{1}{S_{\v p}^{2}}\!=\!A\sum_{p\in\mathfrak{b}}\frac{1}{S_{\v\ell p}^{2}}\!=\!A\cs{\ell\mathfrak{b}}
\]
from which one concludes
\[
\left(\frac{\qd{\ell p}}{\qd p}\right)^{2}=\frac{\cs{\mathfrak{b}}}{\cs{\galpi{\ell}\mathfrak{b}}}
\]
implying Eq.\prettyref{eq:dimratio}, since quantum dimensions are
positive numbers.
\end{singlespace}

If $\mathfrak{g}\!\subseteq\!\mint{\ell}$ is not contained in $\munt$,
then $\mathfrak{g}$ is a $\mathbb{Z}_{2}$-extension of $\gp\!=\!\mathfrak{g}\!\cap\!\munt\!\subseteq\!\munt$,
hence any block $\mathfrak{b}\!\in\!\bl{\mathfrak{g}}$ is either
itself a $\gp$-block, or $\mathfrak{b}\!=\!\mathfrak{b}_{{\scriptscriptstyle +}}\!\cup\!\mathfrak{b}_{{\scriptscriptstyle -}}$
with $\mathfrak{b}_{{\scriptscriptstyle \pm}}\!\in\!\bl{\gp}$ and
$\cs{\mathfrak{b}_{{\scriptscriptstyle \pm}}}\!=\!2\!\cs{\mathfrak{b}}$
(cf. \prettyref{thm:zentextstruct}). In either case Eq.\prettyref{eq:dimratio}
follows by the above argument, since $\ell\mathfrak{b}\!=\!\ell\mathfrak{b}_{{\scriptscriptstyle +}}\!\cup\!\ell\mathfrak{b}_{{\scriptscriptstyle -}}$
and $\cs{\ell\mathfrak{b}_{{\scriptscriptstyle \pm}}}\!=\!\cs{\mathfrak{b}_{{\scriptscriptstyle \pm}}}\!=\!2\!\cs{\mathfrak{b}}\!=\!2\!\cs{\ell\mathfrak{b}}$
by \prettyref{cor:galactbl}.
\end{proof}
\prettyref{lem:dimratio} gives a fairly precise description of the
distribution of quantum dimensions (counted with multiplicity) in
blocks related by Galois permutations, and a similar result for classes
would be most desirable. Supported by extensive computational evidence,
the following seems to hold.
\begin{conjecture}
\label{conj:spect}If $\mathfrak{g}\!\subseteq\!\mint{\ell}$ and
$\ccl\!\in\!\cl{\mathfrak{g}}$ is a $\mathfrak{g}$-class, then the
distribution of quantum dimensions (counted with multiplicity) is
the same in $\ccl$ and $\galg{\ccl}{\ell}$. Put differently, the
univariate polynomial
\begin{equation}
\spp{\ccl}x=\prod_{p\in\ccl}\!\left(x-S_{\v p}^{-2}\right)\label{eq:sppdef}
\end{equation}
satisfies $\spp{\ell\ccl}x\!=\!\spp{\ccl}x$.
\end{conjecture}
\begin{rem}
Note that, as a consequence of \prettyref{lem:inclusion}, it would
be enough to prove \prettyref{conj:spect} for $\mathfrak{g}\!=\!\mint{\ell}$,
since this would imply the general case. Moreover, since $S_{\v p}^{-2}$
is an algebraic integer for each primary $p$, and because $\gal{\ell}\!\left(\spp{\ccl}x\right)\!=\!\spp{\galpi{\ell}\ccl}x$,
the truth of \prettyref{conj:spect} would imply that all the coefficients
of $\spp{\ccl}x$ are rational integers for $\mathfrak{g}\!\in\!\intlat$.
But $\cs{\ccl}\!\in\!\mathbb{Z}$ in case $\spp{\ccl}x\!\in\!\mathbb{Z}\!\left[x\right]$:
in conjunction with \prettyref{conj:algint} this would lead to the
conclusion that $\cs{\ccl}$ is an integer divisor of $\cs{\du{\mathfrak{g}}}$. 
\end{rem}
Finally, let's note that Eq.\prettyref{eq:sppdef} should be contrasted
with the following consequence of \prettyref{lem:dimratio} for $\du{\mathfrak{g}}\!\subseteq\!\mint{\ell}$
(with $\mathfrak{g}$-classes viewed as $\du{\mathfrak{g}}$-blocks)
\begin{equation}
\spp{\galpi{\ell}\ccl}x=\left(\!\frac{\cs{\galpi{\ell}\ccl}}{\cs{\ccl}}\!\right)^{\!\FA{\ccl}}\spp{\ccl}{\!\frac{\cs{\ccl}}{\cs{\galpi{\ell}\ccl}}x\!}\label{eq:spptrans}
\end{equation}

\section{Local sets and twisters\label{sec:Local-sets-and}}

Remember that the $\fc$ $\mathfrak{g}\!\in\!\lat$ is local if $\mathfrak{g}\!\subseteq\!\du{\mathfrak{g}}$.
We'll denote by $\loc$ the set of local $\fc$s; note that, while
the intersection of local $\fc$s is clearly local, this is not necessarily
the case for their join, i.e. $\loc$ is generally not a sublattice
of $\lat$, because it may have several maximal elements. Actually,
$\loc$ is itself a lattice precisely when it has a unique maximal
element.
\begin{lem}
$\mathfrak{g}\!\in\!\lat$ is local iff each $\mathfrak{g}$-class
is a union of $\mathfrak{g}$-blocks, or equivalently, each $\mathfrak{g}$-block
is contained in a well-defined $\mathfrak{g}$-class.
\end{lem}
\begin{proof}
This is a direct consequence of \prettyref{lem:inclusion}, keeping
in mind that $\mathfrak{g}$-blocks are nothing but $\du{\mathfrak{g}}$-classes.
\end{proof}
From the point of view of orbifold deconstruction \cite{Bantay2018a,Bantay2018b},
this is the basic property of local $\fc$s. The point is that the
vacuum block of an orbifold model (the set of primaries originating
in the vacuum primary) is an $\fc$ whose classes correspond to the
different twisted sectors, i.e. collections of twisted modules with
twist elements belonging to the same conjugacy class, while its blocks
correspond to orbits of twisted modules. Since the conjugacy class
of a twist element is the same for all twisted modules on the same
orbit, every block should be included in a well-defined class, hence
the vacuum block should be a local $\fc$ by the above.
\begin{lem}
If $\mathfrak{g}\!\in\!\loc$ and $\bbl{\mathfrak{g}}{\ccl}\!=\!\set{\mathfrak{b}\!\in\!\bl{\mathfrak{g}}}{\mathfrak{b}\!\subseteq\!\ccl}$
denotes the set of $\mathfrak{g}$-blocks contained in the class $\ccl\!\in\!\cl{\mathfrak{g}}$,
then 
\begin{equation}
\FA{\bbl{\mathfrak{g}}{\ccl}\!}\!=\!\sum_{\mathfrak{b}\subseteq\ccl}\!\sp{\mathfrak{b}}{\du{\mathfrak{g}}}\!=\!\sum_{\mathfrak{b}\subseteq\du{\mathfrak{g}}}\!\sp{\mathfrak{b}}{\ccl}\label{eq:blockcount}
\end{equation}
\end{lem}
\begin{proof}
The first equality follows from \prettyref{cor:trivoverlap}, while
the second one from Eq.\prettyref{eq:reciprocity} with $\mathfrak{h}\!=\!\mathfrak{b}\!=\!\du{\mathfrak{g}}$.
\end{proof}
\begin{lem}
\label{lem:localcrit}$\mathfrak{g}\!\in\!\lat$ is local iff $\om{\gamma}\!=\!\om{\alpha}\!\om{\beta}$
for all $\alpha,\beta,\gamma\!\in\!\mathfrak{g}$ such that $N_{\alpha\beta}^{\gamma}\!>\!0$.
\end{lem}
\begin{proof}
This follows from the containment $\mathfrak{g\!\subseteq\!\du{\mathfrak{g}}}$
and \prettyref{cor:trivclass}.
\end{proof}
\begin{cor}
\label{cor:localweights}If the $\fc$ $\mathfrak{g}$ is local, then
$\cw{\alpha}\!\in\!\frac{1}{2}\mathbb{Z}$ for $\alpha\!\in\!\chg$.
\end{cor}
\begin{proof}
According to \prettyref{cor:galactbl}, $\alpha\!\in\!\mathfrak{g}$
implies $\overline{\alpha}\!\in\!\mathfrak{g}$. Since $N_{\alpha\overline{\alpha}}^{\v}\!=\!1$
and $\om{\overline{\alpha}}\!=\!\om{\alpha}$, \prettyref{lem:localcrit}
implies that $\om{\alpha}^{2}\!=\!\om{\v}\!=\!1$, i.e. $\cw{\alpha}\!\in\!\frac{1}{2}\mathbb{Z}$
for $\alpha\!\in\!\chg$.
\end{proof}
Note that the converse is not true: there are many $\fc$s in which
all conformal weights belong to $\frac{1}{2}\mathbb{Z}$, but are
nevertheless not local. On the other hand, the integrality of conformal
weights implies locality by \prettyref{lem:localcrit}, leading to
the following notion \cite{Bantay2018a,Bantay2018b}.
\begin{defn}
A $\CHG$ is an $\fc$ all of whose elements have integer conformal
weight.
\end{defn}
\global\long\def\rcl{\mathtt{\mathsf{R}}}%
\global\long\def\ramcl{\textrm{Ramond~class}}%

\begin{lem}
\label{lem:loccent}Every local $\fc$ $\mathfrak{g}\!\in\!\loc$
has a central class $\rcl\!\in\!\zent{\mathfrak{g}}$ such that $\cech{\rcl}{\alpha}=\om{\alpha}$
for $\alpha\!\in\!\mathfrak{g}$.
\end{lem}
\begin{proof}
By \prettyref{lem:localcrit}
\[
\sum_{\gamma\in\mathfrak{g}}N_{\alpha\beta}^{\gamma}\om{\gamma}\qd{\gamma}=\sum_{\gamma\in\mathfrak{g}}N_{\alpha\beta}^{\gamma}\om{\alpha}\om{\beta}\qd{\gamma}=\om{\alpha}\om{\beta}\qd{\alpha}\qd{\beta}
\]
in case $\alpha,\beta\!\in\!\mathfrak{g}$, i.e. the map $\alpha\mapsto\om{\alpha}\!\qd{\alpha}$
is an irrep of $\svera{\mathfrak{g}}$, hence there does exist a class
$\rcl\!\in\!\cl{\mathfrak{g}}$ such that $\ch{\rcl}\!\left(\alpha\right)\!=\!\om{\alpha}\!\qd{\alpha}$.
Since, according to \prettyref{cor:localweights}, $\om{\alpha}\!=\!\pm1$
for elements of a local $\fc$, \prettyref{lem:centclchar} implies
that the class $\rcl$ is central with $\cech{\rcl}{\alpha}\!=\!\om{\alpha}$.
\end{proof}
\begin{lem}
$\rcl^{2}\!=\!\du{\mathfrak{g}}$ for $\mathfrak{g}\!\in\!\loc$,
and $\mathfrak{g}$ is a twister iff $\rcl\!=\!\du{\mathfrak{g}}$.
\end{lem}
\begin{proof}
$\cech{\rcl}{\alpha}\!=\!\om{\alpha}$ for $\alpha\!\in\!\mathfrak{g}$
by \prettyref{lem:loccent}, hence $\cech{\rcl^{2}}{\alpha}\!=\!\om{\alpha}^{2}\!=\!1$
because of \prettyref{cor:localweights}, proving that indeed $\rcl^{2}\!=\!\du{\mathfrak{g}}$.
On the other hand, $\rcl\!=\!\du{\mathfrak{g}}$ precisely when $\om{\alpha}\!=\!\cech{\rcl}{\alpha}\!=1$
for all $\alpha\!\in\!\mathfrak{g}$, i.e. when $\mathfrak{g}$ is
a $\CHG$.
\end{proof}
\begin{cor}
\label{cor:twistext}Every local $\fc$ $\mathfrak{g}\!\in\!\loc$
is either a $\CHG$ or a $\mathbb{Z}_{2}$-extension of a $\CHG$.
\end{cor}
\begin{proof}
If $\mathfrak{g}\!\in\!\loc$ is not a $\CHG$, then the class $\rcl$
generates a central subgroup  of order $2$, and the corresponding
central quotient $\set{\alpha\!\in\!\mathfrak{g}}{\om{\alpha}\!=\!1}$
is clearly a $\CHG$, proving the assertion.
\end{proof}
We shall call the class $\rcl\!\in\!\zent{\mathfrak{g}}$, whose existence
is guaranteed by \prettyref{lem:loccent}, the $\ramcl$ of the local
$\fc$ $\mathfrak{g}\!\in\!\loc$. The rationale for this nomenclature
is that, in case $\rcl$ differs from the trivial class $\du{\mathfrak{g}}$,
there exists a suitable fermionic generalization of orbifold deconstruction
in which the blocks contained in the trivial class account for the
Neveu-Schwarz (bosonic) sector of the deconstructed model, while those
in the $\ramcl$ describe the fermionic (Ramond) sector.
\begin{lem}
\label{lem:trivblocks}If $\mathfrak{g}\!\in\!\loc$ and $\mathfrak{b}\!\in\!\bl{\mathfrak{g}}$
is a block contained in the central class $\zcl\!\in\!\zent{\mathfrak{g}}$,
then the conformal weights of its elements multiplied by the order
of the product $\zcl\rcl$ differ by integers.
\end{lem}
\begin{proof}
Let $n$ denote the order of the class $\zcl\rcl$ viewed as an element
of the Abelian group $\zent{\mathfrak{g}}$. According to \prettyref{lem:blcrit2},
for any two primaries $p$ and $q$ contained in the block $\mathfrak{b}\!\subseteq\!\zcl$
there exists $\alpha\!\in\!\mathfrak{g}$ such that $N_{\alpha p}^{q}\!>\!0$.
By \prettyref{lem:centralchar} this implies that $\om q\!=\!\om{\alpha}\!\om p\!\cech{\zcl}{\alpha}\inv\!=\!\om p\cech{\zcl\rcl}{\alpha}\inv$,
and because $\cech{\zcl\rcl}{\alpha}^{n}\!=\!\cech{\du{\mathfrak{g}}}{\alpha}\!=\!1$,
this means that $\om q^{n}\!=\!\om p^{n}$, i.e. $n(\cw q\!-\!\cw p)\!\in\!\mathbb{Z}$.
\end{proof}
\begin{cor}
A $\dcl$ is contained in the $\ramcl$ $\rcl$ precisely when the
conformal weights of its elements differ by integers.
\end{cor}
We note that the above results can be generalized further: there is
an intrinsic way to define Adams operations and power maps for local
$\fc$s \cite{Bantay2018a,Bantay2018b}, which lead to a well-defined
notion of the order of classes, and for central classes this coincides
with their multiplicative order (as elements of the center).
\begin{lem}
If $\ccl\!\in\!\cl{\mathfrak{g}}$ is a class of a local $\fc$ $\mathfrak{g}\!\in\!\loc$,
then $\sp{\mathfrak{b}}{\rcl\ccl}\!\geq\!1$ for all blocks $\mathfrak{b}\!\in\!\bl{\mathfrak{g}}$
contained in $\ccl$.
\end{lem}
\begin{proof}
Let $D^{\omega}$ denote the vector with components $D_{p}^{\omega}\!=\!\om p\inv\qd p$
for $p\!\in\!\mathfrak{b}$. According to Eq.\prettyref{eq:verrep},
one has for $\alpha\!\in\!\mathfrak{g}$ and $\mathfrak{b}\!\subseteq\!\ccl$
\begin{gather*}
\sum_{q\in\mathfrak{b}}\fm_{\mathfrak{b}}\!\left(\alpha\right)_{p}^{q}D_{q}^{\omega}\!=\!\sum_{q}N_{\alpha p}^{q}\om q\inv\qd q\!=\!\om{\alpha}\inv\om p\inv\frac{S_{\alpha p}}{S_{\v\v}}\\
\!=\!\om{\alpha}\inv\frac{S_{\alpha p}}{S_{\v p}}\om p\inv\frac{S_{\v p}}{S_{\v\v}}\!=\!\ch{\rcl\ccl}\!\left(\alpha\right)\!D_{p}^{\omega}
\end{gather*}
in case $p\!\in\!\mathfrak{b}$. In other words, $D^{\omega}$ is
a common eigenvector of the matrices $\fm_{\mathfrak{b}}\!\left(\alpha\right)$,
with eigenvalue $\ch{\rcl\ccl}\!\left(\alpha\right)$. Since $D_{p}^{\omega}\!\neq\!0$,
one gets that $\sp{\mathfrak{b}}{\rcl\ccl}$, the multiplicity of
$\ch{\rcl\ccl}$ in $\bfm_{\mathfrak{b}}$, is at least $1$.
\end{proof}
\begin{cor}
The number of blocks contained in the $\ramcl$ equals the number
of blocks contained in the trivial class.
\end{cor}
\begin{proof}
Since $\sp{\mathfrak{b}}{\rcl}\!\leq\!1$ by \prettyref{prop:overlapbound}
(since $\rcl$ is central), while $\sp{\mathfrak{b}}{\rcl}\!\geq\!1$
for $\mathfrak{b}\!\subseteq\!\du{\mathfrak{g}}$ by the above Lemma,
we get $\sp{\mathfrak{b}}{\rcl}\!=\!1$ and
\[
\FA{\bbl{\mathfrak{g}}{\rcl}\!}\!=\!\sum_{\mathfrak{b}\subseteq\du{\mathfrak{g}}}\!\sp{\mathfrak{b}}{\rcl}\!=\!\FA{\bbl{\mathfrak{g}}{\du{\mathfrak{g}}}\!}
\]
by Eq.\prettyref{eq:blockcount}, proving the claim.
\end{proof}
Roughly speaking, the above Corollary says that, in case the $\ramcl$
is non-trivial, there is an equal number of bosonic and fermionic
degrees of freedom in the corresponding deconstructed model.
\begin{thm}
\begin{singlespace}
\label{thm:main} Every local $\fc$ belongs to $\pd{\intlat}$, i.e.
$\qd{\alpha}\!\in\!\mathbb{Z}$ for all $\alpha\!\in\!\mathfrak{g}$.
\end{singlespace}
\end{thm}
\begin{proof}
We shall prove that $\ell\mathfrak{g}\!\subseteq\!\du{\mathfrak{g}}$
for all $\ell\!\in\!\zn N$. Since $\ell\mathfrak{g}\!\subseteq\!\ell\du{\mathfrak{g}}$
as a consequence of locality, and two classes are either equal or
disjoint, this will imply the assertion by \prettyref{lem:galdef}.

In case $\mathfrak{g}$ is a $\CHG$, Eq.\prettyref{eq:omgalpi} gives
$\cw{\galg{\alpha}{\ell}}\!-\!\cw{\galg{\v}{\ell}}\!\in\!\mathbb{Z}$
since $\cw{\alpha}\!\in\!\mathbb{Z}$ for all $\alpha\!\in\!\mathfrak{g}$.
But this is tantamount to $\ell\mathfrak{g}\!\subseteq\!\rcl\!$ by
\prettyref{lem:trivblocks}, and since the $\ramcl$ is trivial for
a $\CHG$, we get that $\ell\mathfrak{g}\!\subseteq\!\du{\mathfrak{g}}$.

If $\mathfrak{g}$ is not a $\CHG$, then $\mathfrak{g}_{{\scriptscriptstyle +}}\!=\!\set{\alpha\!\in\!\mathfrak{g}}{\om{\alpha}\!=\!1}$
is a $\CHG$ by \prettyref{cor:twistext}, and the above argument
shows that $\galpi{\ell}\mathfrak{g}_{{\scriptscriptstyle +}}\!\subseteq\!\du{\mathfrak{g}_{{\scriptscriptstyle +}}}$.
Since $\du{\mathfrak{g}_{{\scriptscriptstyle +}}}\!=\!\du{\mathfrak{g}}\cup\rcl$
by \prettyref{prop:zentext}, this means that either $\galpi{\ell}\mathfrak{g}_{{\scriptscriptstyle +}}\!\subseteq\!\du{\mathfrak{g}}$
or $\galpi{\ell}\mathfrak{g}_{{\scriptscriptstyle +}}\!\subseteq\!\rcl$.
But $\mathfrak{g}_{{\scriptscriptstyle +}}\!\subseteq\!\mathfrak{g}$
implies that $\galpi{\ell}\mathfrak{g}_{{\scriptscriptstyle +}}$
is contained in the $\mathfrak{g}$-block $\ell\mathfrak{g}$, and
because every block of a local $\fc$ is contained in precisely one
class, one has either $\galpi{\ell}\mathfrak{g}\!\subseteq\!\du{\mathfrak{g}}$
or $\galpi{\ell}\mathfrak{g}\!\subseteq\!\rcl$. Since $\mathfrak{g}$
is not a $\CHG$, there exists some $\alpha\!\in\!\mathfrak{g}$ of
half-integral conformal weight, which implies that the conductor $N$
is even, and $\cw{\ell\alpha}\!-\!\cw{\ell\v}$ belongs to $\mathbb{Z}\!+\!\frac{1}{2}$
according to Eq.\prettyref{eq:omgalpi}: this contradicts $\galpi{\ell}\mathfrak{g}\!\subseteq\!\rcl$
as a consequence of \prettyref{lem:trivblocks}, and thus proves that
$\galpi{\ell}\mathfrak{g}\!\subseteq\!\du{\mathfrak{g}}$ in this
case as well.
\end{proof}
As surprising as this result may look at first sight, it is actually
easy to understand from a vantage point of view. Indeed, it follows
from \prettyref{lem:localcrit} that the elements of a local $\fc$
$\mathfrak{g}\!\in\!\loc$ are the simple objects of a symmetric monoidal
subcategory of the modular tensor category associated to the conformal
model. By a result of Deligne \cite{Deligne1990}, such categories
can be identified with the representation category of some (finite)
group $G$, hence the algebra $\svera{\mathfrak{g}}$ associated to
$\mathfrak{g}$ is isomorphic to the character ring of the group $G$,
and in particular the (quantum) dimension of its elements are rational
integers. But this is not the end of the story, for all general properties
of character rings should apply to $\svera{\mathfrak{g}}$ in this
case, so the following statements \cite{Isaacs,Lux-Pahlings} should
hold for $\mathfrak{g}\!\in\!\loc$ and $\alpha\!\in\!\mathfrak{g}$:
\begin{enumerate}
\item the extent of any $\mathfrak{g}$-class is a rational integer dividing
$\cs{\du{\mathfrak{g}}}$;
\item $\alpha\!\left(\ccl\right)\!=\!0$ for some class $\ccl\!\in\!\cl{\mathfrak{g}}$
iff $\qd{\alpha}\!>\!1$;
\item $\FA{\zent{\mathfrak{g}}\!}^{2}\qd{\alpha}^{2}$ is an integer divisor
of $\cs{\du{\mathfrak{g}}}^{2}$;
\item if $\qd{\alpha}^{2}$ is coprime to the ratio ${\displaystyle \frac{\cs{\du{\mathfrak{g}}}}{\cs{\ccl}}}$
for some class $\ccl\!\in\!\cl{\mathfrak{g}}$, then either $\FA{\alpha\!\left(\ccl\right)\negmedspace}\!=\!\qd{\alpha}$
or $\alpha\!\left(\ccl\right)\!=\!0$.
\end{enumerate}
All these assertions are well-known properties of character rings,
e.g. the first one just states that the size of a conjugacy class
is an integer dividing the order of the group, while the third one
is equivalent to Ito's famous theorem \cite{Isaacs}. What is really
amazing is that, as suggested by extensive computational evidence,
they seem to hold for all members of $\intlat$, even in cases when
there is no finite group with a suitable character ring. From this
point of view, it seems fair to say that elements of $\intlat$ describe
``character rings'' of some natural generalization of the group
concept. This interpretation seems the more reasonable as a host of
group theoretical notions may be generalized to arbitrary elements
of $\intlat$: we have already encountered Abelian (\prettyref{def:Abel})
and nilpotent $\fc$s (\prettyref{def:nilp}), but the notion of (super)solvability
can also be generalized to $\intlat$.
\begin{defn}
\label{def:solv}An $\fc$ $\mathfrak{g}\!\in\!\intlat$ is solvable
(supersolvable) if there exists a chain
\[
\left\{ \v\right\} \!=\!\mathfrak{g}_{0}\!\subseteq\!\mathfrak{g}_{1}\!\subseteq\!\cdots\!\subseteq\!\mathfrak{g}_{n}\!=\!\mathfrak{g}
\]
of $\fc$s $\mathfrak{g}_{i}\!\in\!\lat$ such that  $\cs{\mathfrak{g}_{i}}$
equals $\cs{\mathfrak{g}_{i-1}}$ times a prime power (a prime number)
for $i\!=\!1,\ldots,n$.
\end{defn}
For local $\fc$s this is clearly equivalent to the (super)solvability
of the corresponding group, and one may speculate whether some kind
of analogue of the Feit-Thompson theorem holds, i.e. whether $\mathfrak{g}\!\in\!\intlat$
is solvable provided $\cs{\mathfrak{g}}$ is odd. 

\section{Summary and outlook}

As we have seen, fusion closed sets of primaries of a conformal model
(or modular tensor category ) have a fairly deep structure, generalizing
many aspects of the character theory of finite groups. Of course,
this is no accident, since vacuum blocks of orbifold models, which
correspond on general grounds to the character ring of the twist group,
form a special class of $\fc$s. But it turns out that the parallel
with character theory goes much further, even for $\fc$s that have
no group theoretic origin. Many classical notions from group theory
(like nilpotency, solubility, etc.) may be generalized to arbitrary
$\fc$s, and the corresponding properties go over almost verbatim
to this more general setting. In this respect, a major goal of the
present work is to illustrate this close analogy with classical group
theory, but it should be stressed that $\fc$s are more than just
some fancy generalization of the group concept, since they possess
genuinely new properties, as exemplified by the reciprocity relations
Eq.\prettyref{eq:reciprocity} or \prettyref{lem:dimratio}. 

Of course, the results presented fall short of giving due account
of all important aspects of $\fc$s relevant to orbifold deconstruction.
In particular, the $\lambda$-ring structure of local $\fc$s, which
is of utmost importance for the identification of twist groups \cite{Bantay2018a,Bantay2018b},
has not been treated, nor the questions related to inertia groups
of blocks. While fundamental, we felt that the presentation of these
issues could obscure the overall pattern, and have consequently decided
to relegate their discussion to some future work.

Finally, a few words about the mathematics involved. While many of
the results presented might be formulated in the language of (unitary)
modular tensor categories \cite{Turaev,Bakalov-Kirillov,FFRS,Frohlich2009},
in our opinion this could obscure the analogies with group theory,
which were among the major motivations of this work. Besides this,
some of the more interesting results and conjectures, like \prettyref{lem:galpi}
or \prettyref{conj:spect}, seem difficult to formulate using category
theory solely. For these reasons, we opted for a mode of exposition
closer to that of classical texts \cite{Isaacs,Serre} on 
representation theory.

\end{document}